\documentclass{article}
\usepackage[a4paper,top=3cm,bottom=3.5cm,left=3cm,right=3cm,marginparwidth=1.75cm]{geometry}
\usepackage[hidelinks]{hyperref}
\usepackage{amsmath, amssymb, amsthm, bbm, graphicx, cleveref, caption, tikz, xcolor, soul}

\newcommand{\ddp}[2]{\frac{\partial#1}{\partial#2}}

\newcommand{\D}{\partial D}
\renewcommand{\S}{\mathcal{S}}
\newcommand{\K}{\mathcal{K}}
\newcommand{\outside}{\mathbb{R}^3\setminus \overline{D}}

\renewcommand*{\Re}{\operatorname{Re}}
\renewcommand*{\Im}{\operatorname{Im}}

\newcommand{\de}{\: \mathrm{d}}
\newcommand{\R}{\mathbb{R}}
\newcommand{\iu}{\mathrm{i}\mkern1mu}
\newcommand{\C}{\mathcal{C}}
\graphicspath{{images/}}
\renewcommand{\L}{\mathcal{L}}

\renewcommand{\v}{\mathbf{v}}
\newcommand{\w}{\mathbf{w}}
\renewcommand{\k}{\mathbf{k}}
\renewcommand{\P}{\mathcal{P}}

\renewcommand{\S}{\mathcal{S}}
\newcommand{\T}{\mathcal{T}}

\newcommand{\Z}{\mathbb{Z}}
\newcommand{\p}{\partial}
\renewcommand{\epsilon}{\varepsilon}
\newcommand{\dx}{\: \mathrm{d}}

\newcommand{\ie}{\textit{i.e.}}
\newcommand{\nm}{\noalign{\smallskip}}
\newcommand{\ds}{\displaystyle}

\DeclareRobustCommand{\Chi}{{\mathpalette\irchi\relax}}
\newcommand{\irchi}[2]{\raisebox{\depth}{$#1\chi$}}

\newtheorem{thm}{Theorem}
\newtheorem{prop}[thm]{Proposition}
\newtheorem{lemma}[thm]{Lemma}

\newtheorem{remark}[thm]{Remark}

\numberwithin{equation}{section}
\numberwithin{thm}{section}

\title{Exceptional points in parity--time-symmetric subwavelength metamaterials}

\author{Habib Ammari\thanks{\footnotesize Department of Mathematics, ETH Z\"urich, R\"amistrasse 101, CH-8092 Z\"urich, Switzerland (habib.ammari@math.ethz.ch, bryn.davies@sam.math.ethz.ch, erik.orvehed.hiltunen@sam.math.ethz.ch).}\and Bryn Davies\footnotemark[1]  \and Erik Orvehed Hiltunen\footnotemark[1] \and  Hyundae Lee\thanks{\footnotesize  Department of Mathematics, Inha University,  253 Yonghyun-dong Nam-gu,  Incheon 402-751,  Korea (hdlee@inha.ac.kr).} \and Sanghyeon Yu\thanks{\footnotesize Department of Mathematics, Korea University, Seoul 02841, S. Korea (sanghyeon\_yu@korea.ac.kr).}}

\date{}

\begin{document}
	\maketitle
	
\begin{abstract}
	When sources of energy gain and loss are introduced to a wave-scattering system, the underlying mathematical formulation will be non-Hermitian. This paves the way for the existence of exceptional points, where eigenmodes are linearly dependent. The primary goal of this work is to study the existence of exceptional points in high-contrast subwavelength metamaterials. We begin by studying a parity--time-symmetric pair of subwavelength resonators and prove that this system supports asymptotic exceptional points. {These are points at which the subwavelength eigenvalues and eigenvectors coincide at leading order in the asymptotic parameters.} We then investigate further properties of parity--time-symmetric subwavelength metamaterials. First, we study the exotic scattering behaviour of a metascreen {composed of repeating parity--time-symmetric pairs of subwavelength resonators}. We prove that the non-Hermitian nature of this structure means that it exhibits asymptotic unidirectional reflectionless transmission at certain frequencies and demonstrate extraordinary transmission close to these frequencies.
Thereafter,	we consider cavities {containing} many small resonators and use homogenization theory to show that non-Hermitian behaviour can be replicated at the macroscale.  
\end{abstract}
	\vspace{0.5cm}
	\noindent{\textbf{Mathematics Subject Classification (MSC2000):} 35J05, 35C20, 35P20.
		
		\vspace{0.2cm}
		
		\vspace{0.5cm}

	\section{Introduction}
	
	Exceptional points are parameter values at which a system's eigenvalues and their associated eigenvectors simultaneously coincide. This phenomenon has been observed in a variety of quantum-mechanical, optical, acoustic and photonic settings. We will investigate exceptional points in a setting where the underlying system is non-Hermitian, in the hope of finding linearly dependent eigenvectors. A prominent class of non-Hermitian systems where exceptional points are well known to occur are structures with so-called \emph{parity-time} or $\mathcal{PT}$ \emph{symmetry} \cite{feng2017non, heiss2012physics, miri2019exceptional, ozdemir2019parity}. The exceptional points in such systems originate from the fact that the spectrum of a $\mathcal{PT}$-symmetric operator is conjugate symmetric. In this work, we study the occurrence of exceptional points in structures composed of subwavelength resonators. These are material inclusions with parameters that differ greatly from those of the background medium, the large material contrast meaning that they experience resonant behaviour in response to critical wavelengths much greater than their size. Such structures, often known as \emph{subwavelength metamaterials} to highlight their complex microscopic structure, can exhibit exotic scattering properties and appear in a variety of photonic and phononic applications \cite{ammari2018mathematical, kaina2015negative, liu2000locally, ozbay2006plasmonics}.
	
	We begin by studying a pair of high-contrast subwavelength resonators. This two-body system, which is often known as a \emph{dimer}, is known to exhibit two subwavelength resonant modes \cite{ammari2017double}. We examine the case of non-real material parameters, which corresponds to systems with gain and loss (represented, respectively, by positive or negative imaginary parts), see \Cref{figure1}. The geometry and material parameters are chosen so that the structure is $\mathcal{PT}$-symmetric, which means that the structure is symmetric and that the gain on one resonator has the same magnitude as the loss on the other. {This model problem is a mathematically concise characterization of physical models. Structures based on the principles explored in this work have been theoretically and experimentally studied in acoustics} \cite{shi2016accessing,yi2019conditionally, zhu2014acoustics} {as well as in optics and photonics} \cite{chen2017exceptional, miri2019exceptional, ozdemir2019parity}. {In practice, implementing gain in physical systems is difficult (whereas loss arises through damping). This can be achieved either through the use of amplification} \cite{popa2014activemeta,shi2016accessing, zhu2014acoustics} {or by considering a ``passive'' version of the structure, where loss and greater loss are used to model gain and loss, respectively, along with an overall damping factor} \cite{ozdemir2019parity}. For our model problem, we will prove that the resonant modes can be approximated by the eigenstates of a $2\times 2$ matrix, known as the \emph{weighted capacitance matrix} (Theorem \ref{thm:res}). Then, we show that if these parameters are suitably tuned then the two eigenvalues and eigenvectors of the weighted capacitance matrix coincide, giving what we will refer to as an \emph{asymptotic exceptional point} (Theorem \ref{thm:exceptional}).

	{Due to energy input and output in non-Hermitian systems, standard energy conservation relations no longer apply.} This can result in exotic scattering behaviour and ``generalized energy conservation'' relations \cite{li2012conservation}. While being impossible in Hermitian systems, $\mathcal{PT}$-symmetric structures can have frequencies at which the reflection is zero when the wave is impinging from one side, but non-zero when the wave is impinging from the opposite side \cite{lin2011unidirectional,longhi2011invisibility,yuan2019unidirectional}. We will refer to such case as \emph{unidirectional reflectionless transmission}, or \emph{unidirectional reflection} for short. Also, since energy conservation no longer applies, the scattering coefficients are not bounded by unity, and could possibly be very large. We refer to this as \emph{extraordinary transmission}, which has been demonstrated to occur in both optical and acoustic systems \cite{wu2017ultrastrong,yi2019conditionally,zhu2020influence}.
	
	We study unidirectional reflection and extraordinary transmission in an unbounded, $\mathcal{PT}$-symmetric structure at subwavelength frequencies. This structure is composed of periodically repeating $\mathcal{PT}$-symmetric dimers in a thin sheet, a \emph{metascreen} (see Figure \ref{fig:metascreen}). We will show, in particular, that the reflection coefficients approximately vanish for frequencies close to a critical frequency (Theorem \ref{thm:unidir}). Moreover, as the magnitude of the gain and loss increases, there is a shift in these approximate zeros: the zero of one of the reflection coefficients will be shifted upwards and the other will be shifted downwards. Additionally, for a certain magnitude of the gain/loss, extraordinary transmission will occur at the critical frequency. We emphasize that, unlike previous work based on coupled-mode approximations \cite{lin2011unidirectional,rivolta2016sidecoupled} or perturbation theory \cite{yuan2019unidirectional} which are more formal, the methods presented here provide a mathematically rigorous framework for unidirectional reflectionless transmission. Furthermore, the obtained results are valid even in regimes with large gain and loss. 
	
Structures that are poised at an exceptional point have also applications in enhanced sensors. {These sensors can detect a variety of phenomena such as vibrations, changes in temperature and the presence of small particles such as viruses and nanoparticles. Such phenomena will induce changes in the properties of the sensor, which allows them to be measured. Typically, these changes will be proportional to the strength of the perturbation which is measured.} However, in the case of a sensor that is poised at an exceptional point, the higher-order nature of the singularity means that the output will be greatly enhanced. In particular, an $N$\textsuperscript{th}-order exceptional point (where $N$ eigenmodes coincide) will generally lead to an output that scales with the $N$\textsuperscript{th} root of the strength of the perturbation, meaning that it is greatly enhanced for small perturbations \cite{high-order, chen2017exceptional, hodaei2017enhanced, liu2016metrology, wiersig2014enhancing, wiersig2016sensors}.
	
	We will analyse the macroscopic properties of bounded metamaterials composed of a large number of subwavelength resonators with complex material coefficients. In particular, we consider cavities filled with large numbers of small resonators and use homogenization theory to derive effective material properties as the resonators become infinitesimally small. We show that a cavity of resonators with `fixed sign' (\emph{i.e.} all gain or all loss) converges to an effective system whose material parameters retain this property. We also observe that a structure that is $\mathcal{PT}$-symmetric at the microscale will have real-valued material parameters at the macroscale (Theorem \ref{thm:homogenized}).

The results of this paper on exceptional points pave the way for the mathematical analysis of the phenomenon of topologically protected edge states in  systems of subwavelength resonators with gain and loss. In \cite{active}, it is demonstrated that localized edge modes appear in a periodic structure of subwavelength resonators with a defect in the gain/loss distribution. Similarly to the Hermitian case (see, for instance, \cite{ammari2019topological,drouot1,fefferman,fefferman_mms}), these edge modes can be attributed to the winding of the eigenmodes. In the non-Hermitian case the topological invariants fail to be quantized, but can nevertheless predict the existence of localized edge modes. Moreover, the new expansions of the Green's functions obtained in Section \ref{sec:unidir}, which are uniformly valid when the frequency and the quasiperiodicity approach zero,   turn out to be the key for demonstrating Fano-type resonances (see, for instance, \cite{hai}) in periodic systems of subwavelength resonators \cite{fano}.

	\section{Exceptional points of two resonators} \label{sec:exceptional} 
	\begin{figure}
		\centering
		\begin{tikzpicture}
		\draw (-1.5,0) circle (1);
		\draw (1.5,0) circle (1);
		\node at (-2.5,1){$D_1$};
		\node at (2.5,1){$D_2$};
		\node at (-1.5,0){\small $v_1$};
		\node at (1.5,0){\small $v_2$};
		\node at (-0.73,0.3){\small $\delta_1$};
		\node at (0.8,0.3){\small $\delta_2$};
		\node at (0,0.7){\small $v$};
		\end{tikzpicture}
		\caption{Two subwavelength resonators $D_1$ and $D_2$ with wave speeds $v_1, v_2$, and wave speed $v$ in the surrounding material. The contrast between the $i$\textsuperscript{th} resonator and the surrounding material is described by $\delta_i$, {which is assumed to be small}. This system is $\mathcal{PT}$-symmetric if $D_1 = - D_2$ and $v_1^ 2\delta_1 = \overline{v_2^2\delta_2}$. \label{figure1}}
	\end{figure}
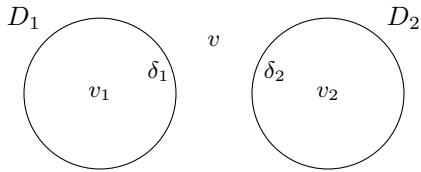
	We will, first, study a structure composed of two resonators $D_1, D_2 \subset \R^3$ which are connected domains such that $\p D_i \in C^{1,s}, 0 < s < 1$. The dimer $D$ is defined as $D= D_1\cup D_2$. We assume that the wave speed $v_i$ inside the $i$\textsuperscript{th} resonator $D_i$ is complex while the wave speed $v$ in the surrounding material is real. Denoting the frequency of the waves by $\omega >0$, we define the wave numbers, for $i=1,2$, as
	$$k = \frac{\omega}{v}, \quad k_i = \frac{\omega}{v_i}.$$	
	We denote the material contrast parameters of the two resonators by $\delta_i, \ i=1,2$, which are also complex-valued and will be assumed to be small in modulus. We study the scattering problem
	\begin{equation} \label{eq:scattering}
	\left\{
	\begin{array} {ll}
	\ds \Delta {u}+ k^2 {u}  = 0 & \text{in } \R^3 \setminus \overline{D}, \\[0.3em]
	\ds \Delta {u}+ k_i^2 {u}  = 0 & \text{in } D_i, \ i=1,2, \\
	\nm
	\ds  {u}|_{+} -{u}|_{-}  = 0  & \text{on } \partial D, \\
	\nm
	\ds  \delta_i \frac{\partial {u}}{\partial \nu} \bigg|_{+} - \frac{\partial {u}}{\partial \nu} \bigg|_{-} = 0 & \text{on } \partial D_i, \ i=1,2, \\
	\nm
	\ds u(x) - u^{\mathrm{in}}(x) & \text{satisfies the Sommerfeld radiation} \\ & \text{condition as }  |x| \rightarrow \infty,
	\end{array}
	\right.
	\end{equation}
	where $|_+$ and $|_-$ denote the limits from the outside and inside of $D$. Here, {$u^{\mathrm{in}}$} is the incident field which we assume satisfies $\Delta u^{\mathrm{in}} + k^2u^{\mathrm{in}} = 0$ in $\R^3$ and the Sommerfeld radiation condition is given by
	\begin{equation}\label{eq:sommerfeld}
	\lim_{|x|\rightarrow \infty}|x| \left(\frac{\p}{\p |x|} -\iu k\right)u = 0,
	\end{equation}
	which corresponds to the case where $u$ radiates energy  outwards (and not inwards).

	Next, we describe the $\P\T$-symmetry of the problem. The parity operator $\mathcal{P}: \R^3 \rightarrow \R^3 $ and the time-reversal operator $\T: \mathbb{C} \rightarrow \mathbb{C}$ are given, respectively, by
	$$\P(x) = -x, \qquad \T(z) = \overline{z}.$$
	We assume that the dimer $D$ is $\P\T$-symmetric, by which we mean that
	$$D_1 = \P D_2, \quad \text{and} \quad  v_1^ 2\delta_1 = \T(v_2^2\delta_2).$$
	We will see in \Cref{sec:cap_matrix} that this is the assumption required to ensure that the {differential system} is $\P\T$-symmetric at leading order. We define $\delta:=|\delta_1|$, and assume that
	$$\delta \ll 1, \quad \delta_2 = O(\delta), \quad v_i = O(1).$$
	The assumption that $\delta$ is small means that we are studying high-contrast resonators. We introduce the notation
	$$v_1^2 \delta_1 := a + \iu{b}, \qquad v_2^2 \delta_2 := a - \iu{b},$$
	for {positive}, real-valued parameters $a$ and $b$. {Under the sign convention fixed by }\eqref{eq:sommerfeld}{, the parameter $b$ represents the magnitude of the gain and loss, which are imposed on $D_1$ and $D_2$ respectively.}
	
Using analytic continuation, one can prove that the resolvent associated with the scattering problem (\ref{eq:scattering}) is meromorphic function of $\omega$ in $\mathbb{C}$.	
	We say that a frequency $\omega$ is a \emph{resonant frequency} if the real part of $\omega$ is positive and there is a non-zero solution to the problem \eqref{eq:scattering} with $u^{\mathrm{in}} = 0$. Moreover, we say that the resonant frequency $\omega$ is a subwavelength resonant frequency if {$\omega$ depends continuously on $\delta$} and $\omega\rightarrow 0$ as $\delta \rightarrow 0$. {We remark that, since the wave speed has order one as $\delta \to 0$, a subwavelength frequency corresponds to a wavelength much larger than the size of $D$ when $\delta$ is small enough.}
	
	We will assume that the frequency $\omega$ scales as $\omega = O(\delta^{1/2})$ when $\delta \to 0$ (in previous work \textit{e.g.} \cite{ammari2018minnaert,ammari2017double}, this was found to be the scaling of the subwavelength resonant frequencies). {In this limit, we will assume that $\nabla u^{\mathrm{in}}\big|_D = O(\omega)$.}
	
	The scattering problem \eqref{eq:scattering} {has been designed as a model problem for subwavelength resonators whose materials contrast greatly with the background medium. This formulation has been used to study the propagation of acoustic waves in bubbly media, in which case $\delta_i$ represents the contrast between the density inside and outside the resonator $D_i$ }\cite{ammari2018minnaert, reviewpeng}. { This is also a useful model problem for gaining intuition into the behaviour of photonic systems. The methods used here can be easily modified to study a two-dimensional version of }\eqref{eq:scattering} { (as in }\cite[Appendix B]{ammari2018minnaert}{), which describes the propagation of polarized electromagnetic waves, see \emph{e.g.}} \cite{ammari2018mathematical, ammari2003electromagnetic}. In the case of electromagnetic waves, $\delta_i$ describes the contrast in either the electric permittivity or magnetic permeability, depending on the polarization.}
	
	\subsection{Layer potential theory on bounded domains} \label{sec:layerpot}
	The solutions to the Helmholtz problem \eqref{eq:scattering} can be effectively studied using representations in terms of integral operators. In particular, let $\S_D^k$ be the single layer potential, defined by
	\begin{equation} \label{eq:Sdef}
	\S_D^k[\phi](x) := \int_{\partial D} G^k(x-y)\phi(y) \dx \sigma(y), \quad x \in \R^3,
	\end{equation}
	where $G^k(x)$ is the outgoing Helmholtz Green's function, given by
	$$
	G^k(x) := -\frac{e^{\iu k|x|}}{4\pi|x|}, \quad x \in \R^3.
	$$
	Here, ``outgoing'' refers to the fact that $G^k$ satisfies the Sommerfeld radiation condition \eqref{eq:sommerfeld}.
	
%	The single layer potential is an operator $\S_{D}^k: L^2(\partial D) \rightarrow H_{\textrm{loc}}^1(\R^3)$ defined by
%	\begin{equation*}
%	\S_D^k[\phi](x) := \int_{\partial D} G^k(x,y)\phi(y) \dx \sigma(y), \quad x \in \R^3.
%	\end{equation*}
%	$H_{\textrm{loc}}^1(\R^3)$ denotes the space of functions that, on every compact subset of $\R^3$, are square integrable and have a square integrable weak first derivative.
	
	For the single layer potential corresponding to the Laplace equation, $\S_D^0$, we will omit the superscript and write $\S_D$. {We will use the same notation for the restriction of $\S_D$ to $\p D$,} which is an operator $\S_D: L^2(\p D) \rightarrow H^1(\p D)$ and is well known to be invertible \cite[Lemma 2.26]{ammari2007polarization}. Here, $H^1(\p D)$ is the space of functions that are square integrable on $\p D$ and have a weak first derivative that is also square integrable.
	
	The Neumann-Poincar\'e operator $\K_D^{k,*}: L^2(\partial D) \rightarrow L^2(\partial D)$ is defined by
	\begin{equation*}\label{eq:K_def}
	\K_D^{k,*}[\phi](x) := \int_{\partial D} \frac{\partial }{\partial \nu_x}G^k(x-y) \phi(y) \dx \sigma(y), \quad x \in \partial D,
	\end{equation*}
	where $\partial/\partial \nu_x$ denotes the outward normal derivative at $x\in\p D$.
	
	The behaviour of $\S_D^k$ on the boundary $\partial D$ is described by the following relations, often known as \emph{jump relations},
	\begin{equation}\label{eq:jump1}
	\S_D^k[\phi]\big|_+ = \S_D^k[\phi]\big|_-,
	\end{equation}
	and
	\begin{equation}\label{eq:jump2}
	\frac{\partial }{\partial \nu}\S_D^k[\phi]\Big|_{\pm}  =  \left(\pm\frac{1}{2} I + \K_D^{k,*}\right) [\phi],
	\end{equation}
	where $|_\pm$ denote the limits from outside and inside $D$. When $k$ is small, the single layer potential satisfies \cite[Lemma 2.43]{ammari2018mathematical}
	\begin{equation} \label{eq:exp_S}
	\S_D^k =\S_D + k\S_{D,1} + O(k^2),
	\end{equation}
	where the error term is with respect to the operator norm $\|.\|_{\L(L^2(\partial D),H^1(\partial D))}$, and the operator $\S_{D,1}: L^2(\partial D) \rightarrow H^1(\partial D)$ is given by
	$$
	\S_{D,1}[\phi](x) = \frac{1}{4\pi \iu} \int_{\D} \phi \dx \sigma, \qquad x\in \D.
	$$
	Moreover, we have \cite[Lemma 2.44]{ammari2018mathematical}
	\begin{equation} \label{eq:exp_K}
	\K_D^{k,*} = \K_D^{0,*} + k^2\K_{D,2} + k^3\K_{D,3} + O(k^4),
	\end{equation}
	where the error term is with respect to the operator norm $\|.\|_{\L(L^2(\partial D),L^2(\partial D))}$ and where 
	$$\K_{D,2}[\phi](x) = \frac{1}{8\pi}\int_{\p D}\frac{(x-y)\cdot \nu_x}{|x-y|}\phi(y)\de \sigma (y), \quad \K_{D,3}[\phi](x) = \frac{\iu}{12\pi}\int_{\p D}(x-y)\cdot \nu_x\phi(y)\de \sigma (y).$$
	
	We also have the following lemma from \cite{ammari2017double}.
	\begin{lemma} \label{lem:ints}
		For any $\varphi\in L^2(\D)$ we have, for $i=1,2$,
		\begin{equation} \label{eq:properties}
		\begin{split}
		\int_{\D_i}\left(-\frac{1}{2}I+\K_D^{*}\right)[\varphi]\de\sigma=0,
		\qquad&\int_{\D_i}\left(\frac{1}{2}I+\K_D^{*}\right)[\varphi]\de\sigma=\int_{\D_i}\varphi\de\sigma,\\
		\int_{\D_i} \K_{D,2}[\varphi]\de\sigma=-\int_{D_i}\S_D[\varphi]\de x, \qquad &\int_{\D_i} \K_{D,3}[\varphi]\de\sigma=\frac{\iu|D_i|}{4\pi}\int_{\D}\varphi\de\sigma.
		\end{split}
		\end{equation}
	\end{lemma}

	A thorough presentation of other properties of the layer potential operators and their use in wave-scattering problems can be found in \emph{e.g.} \cite{ammari2018mathematical, nedelec2001acoustic}. 

	\subsection{Capacitance matrix analysis} \label{sec:cap_matrix}
	
	Our approach to solving \eqref{eq:scattering} is to study the \emph{weighted capacitance matrix}. We will see that the eigenstates of this $2\times2$ matrix characterize, at leading order in $\delta$, the resonant modes of the system. This approach offers a rigorous discrete approximation to the differential problem.
	
	In order to introduce the notion of capacitance, we define the functions $\psi_j$, for $j=1,2$, as 
	\begin{equation} \label{defpsij}
	\psi_j=\S_D^{-1}[\Chi_{\p D_j}],
	\end{equation}
	where $\Chi_A:\mathbb{R}^3\to\{0,1\}$ is used to denote the characteristic function of a set $A\subset\mathbb{R}^3$. The functions $\psi_1$ and $\psi_2$ form a basis for $\ker\left(-\frac{1}{2}I+\K_D^{*}\right)$ {(indeed, this follows from} \eqref{eq:jump2} {and the fact that the constant functions are the Neumann eigenfunctions of $-\Delta$ in $D$ with zero eigenvalue)}. The capacitance coefficients $C_{ij}$, for $i,j=1,2$, are then defined as
	\begin{equation*}
	C_{ij}=-\int_{\D_i} \psi_j\de\sigma,
	\end{equation*}
	and the capacitance matrix is the matrix $C = (C_{ij})$. Finally, we define the \emph{weighted} capacitance matrix $C^v=(C^v_{ij})$ as
	%\begin{equation}
	%C^v_{1j}=v_1^2\delta_1 C_{1j}, \qquad C^v_{2j}={v_2^2\delta_2\right }C_{2,j}.
	%\end{equation}
	%This matrix is an analogue of the capacitance matrix studied in \cite{ammari2017double}, modified to account for the different material parameters inside the different resonators. Note that
	\begin{equation} \label{eq:cap_matrix}
	C^v := V C = \begin{pmatrix}
	v_1^2\delta_1 C_{11} & v_1^2\delta_1 C_{12} \\[0.3em]
	{v_2^2\delta_2}C_{21} & {v_2^2\delta_2}C_{22}
	\end{pmatrix}, \qquad V := \begin{pmatrix} v_1^2\delta_1 & 0 \\ 0 & v_2^2\delta_2 \end{pmatrix}.
	\end{equation}
	This has been weighted to account for the different material parameters inside the different resonators, see \emph{e.g.} \cite{ammari2020close, ammari2017double} for other variants in slightly different settings. It is well known that $C_{21} = C_{12}$, $C_{12} < 0$ and $C_{11} > -C_{12}$, while  the symmetry assumption $D_1 = \P D_2$ implies that $C_{11} = C_{22}$ (see \emph{e.g.} \cite{ammari2017double, diaz2011positivity, lekner2011capacitance}).
		
	We define the functions $S_1^\omega, S_2^\omega$ as	
	$$S_1^\omega(x) = \begin{cases}
	\S_{D}^{k}[\psi_1](x), & x\in\outside,\\
	\S_{D}^{k_i}[\psi_1](x), & x\in D_i, \ i=1,2,\\
	\end{cases} \qquad S_2^\omega(x) = \begin{cases}
	\S_{D}^{k}[\psi_2](x), & x\in\outside,\\
	\S_{D}^{k_i}[\psi_2](x), & x\in D_i, \ i=1,2.\\
	\end{cases} 
	$$
	
	%\begin{thm} \label{thm:res}
	%	As $\delta \rightarrow 0$, the resonance frequencies satisfy the asymptotic formula
	%	$$\omega_i = \sqrt{\frac{\lambda_i}{|D_1|}} + O(\delta), \quad i = 1,2,$$
	%	where $|D_1|$ is the volume of a single resonator. Here, $\lambda_i$ are the eigenvalues of the weighted capacitance matrix $C^v$.
	%\end{thm}
	\begin{lemma} \label{lem:modal}
		As $\omega\rightarrow 0$, the solution to the scattering problem \eqref{eq:scattering} can be written as
		\begin{equation*} 
		u-u^{\mathrm{in}} = q_1S_1^\omega + q_2S_2^\omega - \S_D^k\left[\S_D^{-1}[u^{\mathrm{in}}]\right] + O(\omega),
		\end{equation*} 
		for constants $q_1$ and $q_2$ which satisfy the problem
		\begin{equation}\label{eq:eval_C}
		\left(C^v-\omega^2|D_1|I\right)\begin{pmatrix}q_1\\q_2\end{pmatrix}
		=-
		\begin{pmatrix} {v_1^2\delta_1}\int_{\D_1}\S_D^{-1}[u^{\mathrm{in}}]\de\sigma \\[0.3em]
		{v_2^2\delta_2}\int_{\D_2}\S_D^{-1}[u^{\mathrm{in}}]\de\sigma \end{pmatrix} + O(\delta \omega+\omega^3).
		\end{equation}
	\end{lemma}
	\begin{proof}
		The solutions can be represented as 
		\begin{equation} \label{eq:layer_potential_representation}
		u = \begin{cases}
		u^{\mathrm{in}}(x)+\S_{D}^{k}[\psi](x), & x\in\outside,\\
		\S_{D}^{k_i}[\phi](x), & x\in D_i, \ i=1,2,
		\end{cases}
		\end{equation} 
		for some surface potentials $(\phi,\psi)\in L^2(\D)\times L^2(\D)$, which must be chosen so that $u$ satisfies the transmission conditions across $\D$.
		%This is equivalent to satisfying
		%\begin{equation} \label{eq:A_matrix_equation}
		%\A(\omega,\delta)
		%\begin{pmatrix} \phi_1 \\ \phi_2 \\ \psi_1 \\ \psi_2 \end{pmatrix} = \begin{pmatrix} u^{\mathrm{in}} \\ \delta\ddp{u^{\mathrm{in}}}{\nu_x}\end{pmatrix},
		%\end{equation}
		%where	
		%The densities corresponding to a resonant mode lie in the null space of the operator
		%%\begin{equation*} \label{eq:A_matrix_defn}
		%\A(\omega,\delta):=
		%\left[ {\begin{array}{cccc}
		%	\S_{D}^{\omega/v_1} & 0 & -\S_{D}^{\omega/v} \\
		%	0 & \S_{D}^{\omega/{v_2}} & -\S_{D}^{\omega/v} \\
		%	-\frac{1}{2}I+\K_{D}^{\omega/v_1,*} & 0 & -\delta(\frac{1}{2}I+\K_{D}^{{\omega/v},*}) \\
		%	0 & -\frac{1}{2}I+\K_{D}^{\omega/{v_2},*} & -\delta(\frac{1}{2}I+\K_{D}^{{\omega/v},*})
		%	\end{array} } \right],
		%\end{equation*}
		Using the jump conditions \eqref{eq:jump1} and \eqref{eq:jump2}, we see that in order to satisfy the transmission conditions, the layer densities $\phi$ and $\psi$ must satisfy
		\begin{align*}
		\S_{D}^{k_i}[\phi]-\S_{D}^{k}[\psi]=u^{\mathrm{in}} \quad \text{on}\ \D_i, \\
		\left(-\frac{1}{2}I+\K_{D}^{k_i,*}\right)[\phi]-\delta_i\left(\frac{1}{2}I+\K_{D}^{k,*}\right)[\psi]=\delta_i \ddp{u^{\mathrm{in}}}{\nu} \quad \text{on} \ \D_i,
		\end{align*}
		for $i=1,2$, where $I$ is the identity operator on $L^2(\D)$.	
		% \begin{equation*}  \label{eq:boundary_quasi}
		%\mathcal{A}^\alpha(\omega, \delta)[\Phi^\alpha] =0,  
		%\end{equation*}
		%where
		%\begin{equation} \label{eq:A_quasi_defn}
		%\A^\alpha(\omega,\delta) = 
		%\begin{pmatrix}
		%\hat\S_{D}^{\omega} & -\S_{D}^{\alpha,\omega} \\
		%-\frac{1}{2}I + \hat\K_{D}^{\omega,*} & -\delta\left(\frac{1}{2}I + \left(\mathcal{K}_D^{ -\alpha,\omega}\right)^*\right) \\
		%\end{pmatrix}, \qquad \Phi^\alpha =  \begin{pmatrix}
		%\phi^{\alpha,i} \\
		%\phi^{\alpha,o}
		%\end{pmatrix},
		%\end{equation}
		%and
		%\begin{equation} \label{add1} \hat{\S}_D^\omega = \begin{pmatrix}
		%\S_{D_1}^{\omega} & 0 \\ 0 & \S_{D_2}^{\omega} \end{pmatrix}, \qquad \hat{\K}_D^{\omega,*} = \begin{pmatrix}
		%\K_{D_1}^{\omega,*} & 0 \\ 0 & \K_{D_2}^{\omega,*} \end{pmatrix}.\end{equation}
		From the asymptotic expansions \eqref{eq:exp_S} and \eqref{eq:exp_K} and the assumption that $\nabla u^{\mathrm{in}} \big|_D = O(\omega)$ we have that 
		\begin{gather} 
		% \begin{split}
		\S_D[\phi-\psi]=u^{\mathrm{in}}+O(\omega) \quad \text{on} \ \D_1\cup\D_2, \label{eq:first}\\
		\left(-\frac{1}{2}I+\K_D^*+\frac{\omega^2}{v_i^2}\K_{D,2}\right)[\phi]-\delta_i\left(\frac{1}{2}I+\K_D^*\right)[\psi]=O(\delta\omega+\omega^3) \quad \text{on} \ \D_i. \nonumber
		%\end{split}
		\end{gather}
		From \eqref{eq:first} and the fact that $\S_D$ is invertible we can see that
		\begin{equation}\label{eq:psi}
		\psi=\phi-\S_D^{-1}[u^{\mathrm{in}}]+O(\omega).
		\end{equation} Thus, we are left with the equations
		\begin{align} %\label{eq:second}
		\left(-\frac{1}{2}I+\K_D^*+\frac{\omega^2}{{v}_i^2}\K_{D,2}-\delta_i\left(\frac{1}{2}I+\K_D^*\right)\right)[\phi]=-\delta_i \left(\frac{1}{2}I+\K_D^*\right)\S_D^{-1}[u^{\mathrm{in}}]+O(\delta\omega+\omega^3), \label{eq:phi}
		\end{align}
		on $\D_i, i=1,2$. Integrating \eqref{eq:phi} over $\D_i$, and using \Cref{lem:ints} gives us that
		\begin{equation*}
		-\omega^2\int_{D_i}\S_D[\phi]\de x -{v}_i^2\delta_i\int_{\D_i}\phi\de\sigma=-{v_i^2\delta_i}\int_{\D_i}\S_D^{-1}[u^{\mathrm{in}}]\de\sigma+O(\delta\omega+\omega^3).  \label{eq:D}
		\end{equation*}
		At leading order, \eqref{eq:phi} says that $\left(-\frac{1}{2}I+\K_D^{*}\right)[\phi]=0$ so, since $\psi_1$ and $\psi_2$, defined by (\ref{defpsij}), form a basis for $\ker\left(-\frac{1}{2}I+\K_D^{*}\right)$ (see, for instance, \cite{ammari2018mathematical}), the solution can be written as
		\begin{equation} \label{eq:psi_basis}
		\phi=q_1\psi_1+q_2\psi_2+O(\omega^2+\delta),
		\end{equation}
		for constants $q_1,q_2=O(1)$. Making this substitution we reach the problem
		\begin{equation}\label{eq:eval_C_proof}
		\left(C^v-\omega^2|D_1|I\right)\begin{pmatrix}q_1\\q_2\end{pmatrix}
		=-
		\begin{pmatrix} {v_1^2\delta_1}\int_{\D_1}\S_D^{-1}[u^{\mathrm{in}}]\de\sigma \\[0.3em]
		{v_2^2\delta_2}\int_{\D_2}\S_D^{-1}[u^{\mathrm{in}}]\de\sigma \end{pmatrix} + O(\delta \omega+\omega^3).
		\end{equation}
		The result now follows from \eqref{eq:layer_potential_representation} combined with the expressions for $\phi, \psi$ in \eqref{eq:psi}, \eqref{eq:psi_basis} and \eqref{eq:eval_C_proof}.
	\end{proof}
	
	\begin{thm} \label{thm:res}
		As $\delta \rightarrow 0$, the subwavelength resonant frequencies satisfy the asymptotic formula
		$$\omega_i = \sqrt{\frac{\lambda_i}{|D_1|}} + O(\delta), \quad i = 1,2,$$
		where $|D_1|$ is the volume of a single resonator and the branch of the square root is chosen with positive real part. Here, $\lambda_i$ are the eigenvalues of the weighted capacitance matrix $C^v$.
	\end{thm}
	\begin{proof}
		If $u^{\mathrm{in}} = 0$, we find from \Cref{lem:modal} that there is a non-zero solution $q_1,q_2$ to the eigenvalue problem \eqref{eq:eval_C} precisely when $\omega^2|D_1|$ is an eigenvalue of $C^v$.
	\end{proof}

%	The eigenvalues of the matrix $C^v$ are given by
%	\begin{equation}\label{eq:eval}
%	\lambda_i={a} C_{11} + (-1)^i \sqrt{{a}^2C_{12}^2-{b}^2(C_{11}^2 - C_{12}^2)}.
%	\end{equation}
%	The following theorem describes the asymptotic exceptional point of the resonator dimer, which occurs when $\lambda_1=\lambda_2$.
%	\begin{thm}\label{thm:exceptional}
%		There is a magnitude ${b}_0 = {b}_0({a}) > 0$ of the gain/loss such that the resonator dimer has an asymptotic exceptional point: the frequencies $\omega_1$ and $\omega_2$, and corresponding eigenmodes, coincide to leading order in $\delta$, as $\delta\to0$. Moreover, to leading order in $\delta$, we have
%		\begin{align*}
%		&\text{Case } {b} < {b}_0: \qquad \omega_1 \text{ and } \omega_2 \text{ are real, and } \omega_1 < \omega_2 +O(\delta), \\ 
%		&\text{Case } {b} > {b}_0: \qquad \omega_1 \text{ and } \omega_2 \text{ are non-real, and } \omega_1 = \overline{\omega_2}+O(\delta),
%		\end{align*}
%	as $\delta \to 0$.
%	\end{thm}

	Let $\lambda_i, \v_i = \left(\begin{smallmatrix} \v_i^1 \\ \v_i^2\end{smallmatrix} \right)$ denote an eigenpair of $C^v$. From \Cref{thm:res}, and from \eqref{eq:psi_basis} and \eqref{eq:layer_potential_representation}, it follows that the subwavelength resonant frequencies and corresponding eigenmodes satisfy the asymptotic formulas
	$$\omega_i = \omega_1^{(0)} + O(\delta), \quad u_i = u_i^{(0)} + O(\delta^{1/2}),\quad\text{as }\delta\to0,$$ 
	where
	$$\omega_1^{(0)} := \sqrt{\frac{\lambda_i}{|D_1|}}, \quad u_i^{(0)} := \v_i^{1}S_1^\omega + \v_i^{2}S_2^\omega.$$
	We can compute that the eigenvalues of the matrix $C^v$ are given by
	\begin{equation}\label{eq:eval}
	\lambda_i={a} C_{11} + (-1)^i \sqrt{{a}^2C_{12}^2-{b}^2(C_{11}^2 - C_{12}^2)}.
	\end{equation}
	We then have the following theorem on the dimer's asymptotic exceptional points, which occur when $\lambda_1=\lambda_2$.
	\begin{thm}\label{thm:exceptional}
	There is a magnitude ${b}_0 = {b}_0({a}) > 0$ of the gain/loss such that the resonator dimer has an asymptotic exceptional point. Specifically, at $b=b_0$,
	$$\omega_1^{(0)} = \omega_2^{(0)}, \qquad u_1^{(0)} = Ku_2^{(0)},$$ 
	for some constant $K$, while
	\begin{align*}
		&\text{if } {b} < {b}_0: \qquad \omega_1^{(0)} \text{ and } \omega_2^{(0)} \text{ are real, and } \omega_1^{(0)} < \omega_2^{(0)}, \\ 
		&\text{if } {b} > {b}_0: \qquad \omega_1^{(0)} \text{ and } \omega_2^{(0)} \text{ are non-real, and } \omega_1^{(0)} = \overline{\omega_2^{(0)}}.	
	\end{align*}
\end{thm}
	\begin{proof}
		Combining \Cref{thm:res} and \eqref{eq:eval}, we find that ${b}_0$ is given by
		$${b}_0 = \frac{{a} C_{12}}{\sqrt{C_{11}^2 - C_{12}^2}},$$
		which corresponds to the point where $C^v$ has a double eigenvalue corresponding to a one-dimensional eigenspace. The remaining statements are straightforward to check.
	\end{proof}
\begin{remark}
	\Cref{thm:exceptional} states that the exceptional point occurs only at leading order in $\delta$. This is not due to a limitation of the method and we do not, in fact, expect the system to exhibit an exact exceptional point. This is a consequence of the radiation condition, which means the differential operator corresponding to the problem \eqref{eq:scattering} is not $\mathcal{PT}$-symmetric (even in the case $b=0$ the resonant frequencies have small but non-zero imaginary parts \cite{ammari2017double}). However, the discrete approximation given by the weighted capacitance matrix is indeed $\mathcal{PT}$-symmetric. The approximate nature of the exceptional point can be observed from the simulations presented in \Cref{fig:exceptional}. 
\end{remark}
	\begin{lemma}
		If ${b} \neq {b}_0$, the eigenmodes $u_i$ corresponding to the resonant frequencies $\omega_i$, for $i=1,2$, are given by 
		$$u_i = \v_i^{1}S_1^\omega + \v_i^{2}S_2^\omega  + O(\delta^{1/2}),$$
		as $\delta \to 0$, where $\v_i = \begin{pmatrix}\v_i^{1} & \v_i^{2}\end{pmatrix}^\mathrm{T}$ (using superscript $ \mathrm{T}$ to denote the matrix transpose) are the eigenvectors of $C^v$ given by
		$$\v_i = \begin{pmatrix} -C_{12} \\ C_{11} - \mu_i \end{pmatrix}, \qquad \mu_i = \frac{\lambda_i}{({a}+i{b})}.$$
	\end{lemma}

	\begin{figure}
		\begin{center}
			\includegraphics[width=0.7\linewidth]{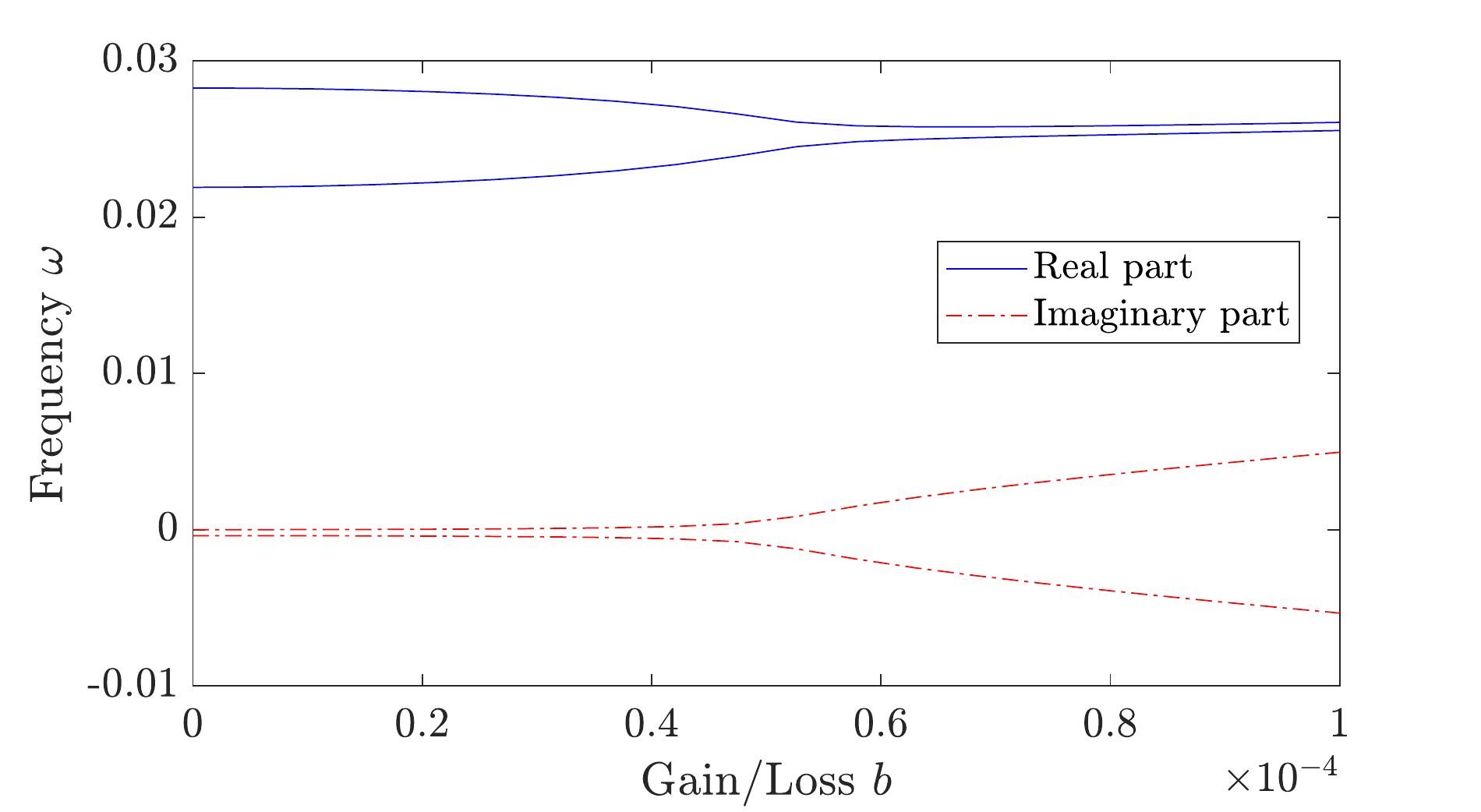}
		\end{center}
		\caption{Plot of the real part (blue) and imaginary part (red) of the resonant frequencies of the dimer as the gain/loss parameter ${b}$ increases. The asymptotic exceptional point occurs at $b_0 \approx 0.5 \times 10^{-4}$, at which point the frequencies coincide to leading order. For $b$ smaller than $b_0$, the frequencies are real, while for $b$ larger than $b_0$ the frequencies are conjugate to each other, again to leading order. Here, we simulate spherical resonators with unit radius, separation distance 2 and material parameters $a = 2\times 10^{-4}$ and $v=1$.} \label{fig:exceptional}
	\end{figure}
	
	\subsection{Numerical computations}
	
	\Cref{fig:exceptional} shows the two resonant frequencies $\omega_1$ and $\omega_2$ as functions of $b$. For $b=b_0$, the resonant frequencies coincide at leading order in $\delta$. The leading-order terms are real for $b<b_0$ and complex conjugates with zero real part for larger $b$. These numerical simulations were performed on spherical resonators using the multipole expansion method, which is outlined in \cite[Appendix~A]{ammari2019topological}.

	\section[PT-symmetric metascreens]{$\mathcal{PT}$-symmetric metascreens} \label{sec:unidir}
	Here, we study a metascreen consisting of periodically repeated $\mathcal{PT}$-symmetric dimers. There are multiple goals. First, we will derive results analogous to those in \Cref{sec:exceptional}, which characterize the band structure and exceptional points of the metascreen. Thereafter, we will solve the plane-wave scattering problem for the metascreen. Using this, we will prove that the metascreen exhibits asymptotic unidirectional reflectionless transmission. In other words, there are frequencies at which an incoming wave from one side will have zero reflection at leading order, while an incoming wave from the opposite side has non-zero reflection. Moreover, we will demonstrate that at a specific magnitude of the gain/loss, the peak transmittance will be extraordinarily large.
	
	\subsection{Scattering problem for the metascreen} \label{sec:scatter}
	\begin{figure}
		\centering
		\includegraphics[width=0.7\linewidth]{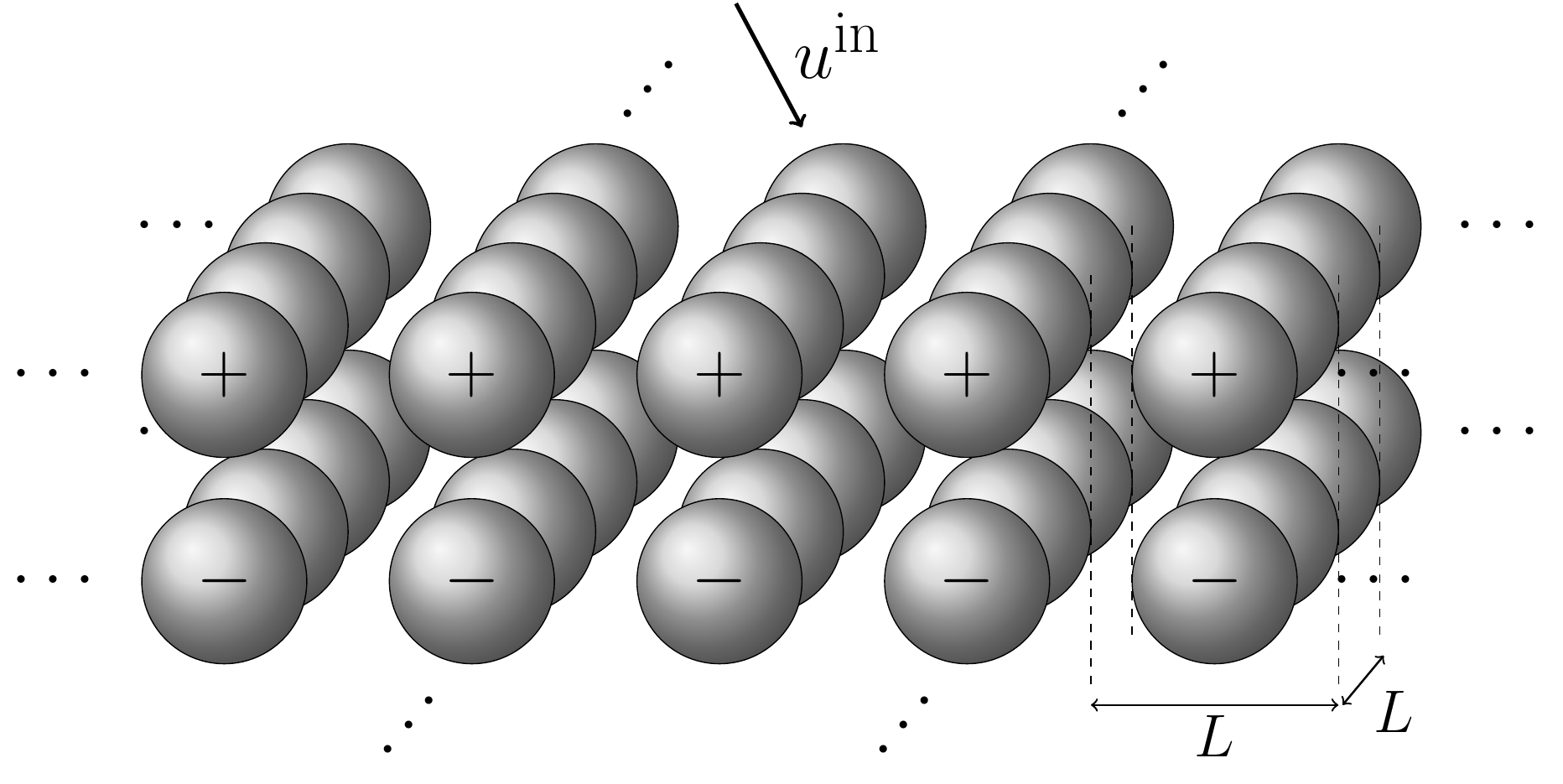}
		\caption{A sketch of a $\mathcal{PT}$-symmetric metascreen with an incident plane wave $u^{\mathrm{in}}$. Here, $+$ and $-$ denote opposite signs of the imaginary part of the material coefficients.} \label{fig:metascreen}
	\end{figure}
	We consider a structure composed of $\mathcal{PT}$-symmetric dimers in a two dimensional square lattice with period $L > 0$. The lattice is given by $\Lambda := L\Z^2$ and we assume that the structure is periodic with unit cell $Y=[-L/2,L/2]\times[-L/2,L/2]\times \R$. We adopt the notation from \Cref{sec:exceptional} where $D$ is a pair of resonators $D_i$ with material parameters $v_i^2\delta_i$, for $i=1,2$:
	$$D = D_1\cup D_2, \quad v_1^2\delta_1 = a+\iu b, \quad v_2^2\delta_2 = a-\iu b.$$
	Additionally, we now assume that $D$ is contained inside $Y$.
	We define the periodically repeated resonators as
	$$\C_i = \bigcup_{(m_1,m_2)\in \Lambda} D_i + (m_1,m_2,0), \ i =1,2, \qquad \C = \C_1 \cup \C_2.$$
	The metascreen is sketched in \Cref{fig:metascreen}.

	The dual lattice $\Lambda^*$ of $\Lambda$ is defined as $\Lambda^* = (2\pi/L)\Lambda$.	The torus $Y^*:=\R^2/\Lambda^*$ is known as the \emph{Brillouin zone}. A function $f(y)$, $y \in \R^2$, is said to be $\alpha$-quasiperiodic, with \emph{quasiperiodicity} $\alpha\in Y^*$, if $e^{-\iu \alpha \cdot y}f(y)$ is periodic as a function of $y$. 
	
	We study the scattering problem 
	\begin{equation} \label{eq:scattering_quasi}
	\left\{
	\begin{array} {ll}
	\ds \Delta {u}+ k^2 {u}  = 0 & \text{in } \R^3 \setminus \C, \\[0.3em]
	\ds \Delta {u}+ k_i^2 {u}  = 0 & \text{in } \C_i, \ i=1,2, \\
	\nm
	\ds  {u}|_{+} -{u}|_{-}  = 0  & \text{on } \partial \C, \\
	\nm
	\ds  \delta_i \frac{\partial {u}}{\partial \nu} \bigg|_{+} - \frac{\partial {u}}{\partial \nu} \bigg|_{-} = 0 & \text{on } \partial \C_i, \ i=1,2, \\
	\nm
	\ds u(x) - u^{\mathrm{in}}(x) & \text{satisfies the outgoing $\alpha$-quasiperiodic} \\ & \text{radiation condition as }  |x_3| \rightarrow \infty.
	\end{array}
	\right.
	\end{equation}
	Here, $u^{\mathrm{in}}$ is the incident field and the outgoing quasiperiodic radiation condition loosely states that $u(x) - u^{\mathrm{in}}(x)$ behaves as a superposition of outgoing plane waves sufficiently far away from the metascreen. {More specifically}, for $|x_3| > M$ for some $M > 0$, $u-u^{\mathrm{in}}$ a Rayleigh expansion given by \cite{bao, botten2013electromagnetic}
	$$u(x)-u^{\mathrm{in}}(x) = \begin{cases} \ds
		\sum_{q\in \Lambda^*} A_{q,+}e^{\iu \k_{q,+}\cdot x}, & x_3 > M, \\[1em]
		\ds \sum_{q\in \Lambda^*} A_{q,-}e^{\iu \k_{q,-}\cdot x}, & x_3 < M,\end{cases}$$
	for some coefficients $A_{q,+}$, $A_{q,-}$, where, for $q=(q_1,q_2)$ and $\alpha = (\alpha_1,\alpha_2)$,
	$$\k_{q,+} = \begin{pmatrix}\alpha_1 + q_1 \\ \alpha_2+q_2 \\ \sqrt{k^2 - |\alpha+q|^2}
	\end{pmatrix}, \qquad \k_{q,-} = \begin{pmatrix}\alpha_1 + q_1 \\ \alpha_2+q_2 \\ -\sqrt{k^2 - |\alpha+q|^2}
	\end{pmatrix}.$$
	Here, the square-root is chosen with positive imaginary part. We seek solutions $u$ which are $\alpha$-quasiperiodic in $(x_1,x_2)$ for some $\alpha$, \ie{}
	$$u(x + (m_1,m_2,0)) = e^{\iu\alpha\cdot (m_1,m_2)}u(x), \quad (m_1,m_2)\in \Lambda.$$
	If $u^{\mathrm{in}} = 0$, these $\alpha$-quasiperiodic solutions are the Bloch modes of the metascreen, while if $u^{\mathrm{in}}$ is a plane wave, we will seek solutions at $\alpha$ specified by the wave vector of $u^{\mathrm{in}}$. 
	
	We will study the scattering problem \eqref{eq:scattering_quasi} using a layer potential formulation analogously as in \Cref{sec:exceptional}.  For $\alpha\in Y^*$, the quasiperiodic Green's function $G^{\alpha,k}(x)$ is defined as the solution to
	$$\Delta G^{\alpha,k}(x) + k^2G^{\alpha,k}(x) = \sum_{(m_1,m_2) \in \Lambda} \delta(x-(m_1,m_2,0))e^{\iu \alpha\cdot (m_1,m_2)},$$
	along with the outgoing quasiperiodic radiation condition,  
	where $\delta(x)$ denotes the Dirac delta distribution. In this work, we will mostly focus on the case $k \neq |\alpha+q|$ for all $q  \in \Lambda^*$, in which $G^{\alpha,k}$ can be written as
	\begin{equation}\label{eq:Gquasi}
	G^{\alpha,k}(x,y) := -\sum_{(m_1,m_2) \in \Lambda} \frac{e^{\iu k|x- (m_1,m_2,0)|}}{4\pi|x-(m_1,m_2,0)|}e^{\iu \alpha \cdot (m_1,m_2)},
	\end{equation}
	where the series in the spatial representation \eqref{eq:Gquasi} converges uniformly for $x$ in compact sets of $\R^3$, $x\neq 0$ (see \emph{e.g} \cite[Section 2.12]{ammari2018mathematical}, or \cite{petit2013electromagnetic} for a more general review). Then, again assuming that $k \neq |\alpha+q|$ for all $q  \in \Lambda^*$, we define the quasiperiodic single layer potential $\mathcal{S}_D^{\alpha,k}$ by
	$$\mathcal{S}_D^{\alpha,k}[\phi](x) := \int_{\partial D} G^{\alpha,k} (x-y) \phi(y) \dx\sigma(y),\quad x\in \mathbb{R}^3.$$
	On the boundary of $D$, it satisfies the jump relations
	\begin{equation} \label{eq:jump1_quasi}
	\S_D^{\alpha,k}[\phi]\big|_+ = \S_D^{\alpha,k}[\phi]\big|_- \quad \mbox{on}~ \p D,
	\end{equation}
	and
	\begin{equation} \label{eq:jump2_quasi}
	\frac{\p}{\p\nu} \mathcal{S}_D^{\alpha,k}[\phi] \Big|_{\pm} = \left( \pm \frac{1}{2} I +( \mathcal{K}_D^{-\alpha,k} )^*\right)[\phi]\quad \mbox{on}~ \p D,
	\end{equation}
	where $(\mathcal{K}_D^{-\alpha,k})^*$ is the quasiperiodic Neumann-Poincar\'e operator, given by
	$$ (\mathcal{K}_D^{-\alpha, k} )^*[\phi](x):= \int_{\p D} \frac{\p}{\p\nu_x} G^{\alpha,k}(x-y) \phi(y) \dx\sigma(y).$$
	\begin{lemma} \label{lem:inv}
		The quasiperiodic single layer potential
		$\mathcal{S}_D^{\alpha,k} : L^2(\p D) \rightarrow H^1(\p D)$ is invertible if $k$ is small enough and $k \neq  |\alpha + q|$ for all $q \in \Lambda^*$.
	\end{lemma}
	\begin{proof}
		If $\varphi \in L^2(\p D)$ satisfies $\S_D^{\alpha,k}[\varphi] = 0$ on $\p D$, then $u := \S_D^{\alpha,k}[\varphi]$ satisfies $\Delta u + k^2u = 0$ in $Y\setminus \p D$. Since $0$ is not a Dirichlet eigenvalue of $-\Delta$ in $D$, and neither a Dirichlet eigenvalue of $-\Delta$ on $Y\setminus D$ with quasiperiodic conditions on $\p Y$, it follows that $u = 0$ for small enough $k$. Then, from the jump condition \eqref{eq:jump1_quasi} we have $\varphi = \p u/ \p \nu|_+ - \p u/ \p \nu|_- = 0$, which proves the claim.
	\end{proof}
	\begin{remark} \label{rmk:2D}
		Throughout \Cref{sec:unidir}, we study the problem in three spatial dimensions. However, all the arguments carry over to the case of two spatial dimensions with a one-dimensional screen of resonators, producing similar results. Indeed, the numerical simulations used to create \Cref{fig:bandexceptional,fig:scattering} are performed on an arrays of circular resonators that are the two-dimensional analogues of those which are analysed here.
	\end{remark}
	
	\subsection{Band structure and exceptional points}\label{sec:band}
	In this section, we study the resonance problem, or in other words, the problem \eqref{eq:scattering_quasi} with $u^{\mathrm{in}} = 0$. Moreover, we study the regime when $\omega\to 0$ while $|\alpha| > c > 0$ for some $c$ independent on $\omega$. By modifying the arguments used to derive (A.6) and (A.7) of \cite{ammari2017subwavelength}, we can obtain asymptotic expansions for a three-dimensional structure with two-dimensional periodicity. In particular, we have that as $k\to0$
	\begin{equation}\label{eq:exp_S_quasi}
	\S_D^{\alpha,k} =\S_D^{\alpha,0} + O(k^2),
	\end{equation}
	and
	\begin{equation}\label{eq:exp_K_quasi}
	(\mathcal{K}_D^{-\alpha, k} )^* = (\mathcal{K}_D^{-\alpha, 0} )^* + O(k^2).
	\end{equation}
	Here, the error terms are with respect to the operator norms, and are uniform for $|\alpha| > c > 0$. As in \Cref{sec:exceptional}, we define the quasiperiodic capacitance coefficients $C_{ij}^\alpha$, for $i,j=1,2$, as
	\begin{equation}
	C_{ij}^\alpha=-\int_{\D_i} \psi_j^\alpha\de\sigma, \qquad \psi_j^\alpha=(\S_D^{\alpha,0})^{-1}[\Chi_{\p D_j}].
	\end{equation}
	The \emph{quasiperiodic capacitance matrix} $C^\alpha$ is defined as the matrix $C^\alpha = (C_{ij}^\alpha)$, while the \emph{weighted} quasiperiodic capacitance matrix $C^{v,\alpha}$ is defined as
	\begin{equation*}
	C^{v,\alpha}=VC^{\alpha},
	\end{equation*}
	with $V$ as in \eqref{eq:cap_matrix}. As we shall see, the capacitance matrix gives the leading-order approximation of the solution to the resonance problem \eqref{eq:scattering_quasi}. We have the following lemma from \cite[Lemma 3.1]{ammari2019topological}.
	\begin{lemma}
		We have
		$$C_{11}^\alpha = C_{22}^\alpha \in \R, \qquad C_{12}^\alpha = \overline{C_{21}^\alpha}.$$
	\end{lemma}
	Directly following the arguments of \Cref{sec:exceptional}, but instead using the jump conditions \eqref{eq:jump1_quasi}, \eqref{eq:jump2_quasi} and the asymptotic expansions \eqref{eq:exp_S_quasi}, \eqref{eq:exp_K_quasi}, we can show the following theorem on the band structure of the metascreen, which is the analogue of \Cref{thm:res}.
	\begin{thm} \label{thm:res_quasi}
		As $\delta \rightarrow 0$, the quasiperiodic resonant frequencies satisfy the asymptotic formula
		$$\omega_i^\alpha = \sqrt{\frac{\lambda_i^\alpha}{|D_1|}} + O(\delta^{3/2}), \quad i = 1,2,$$
		where $|D_1|$ is the volume of a single resonator. Here, $\lambda_i^\alpha$ are the eigenvalues of the weighted quasiperiodic capacitance matrix $C^{v,\alpha}$.
	\end{thm}
	
	\begin{remark} \label{rmk:asymdiff}
		The error in \Cref{thm:res_quasi} is an order of magnitude smaller than in \Cref{thm:res}, which was the equivalent result for a finite system of just two resonators. This is a consequence of the fact that the $O(k)$ term in the expansion (\ref{eq:exp_S_quasi}) of the quasiperiodic single layer potential is zero.
	\end{remark}

	Analogously to the case of a single dimer studied in \Cref{sec:exceptional}, the eigenvalues of the weighted quasiperiodic capacitance matrix are given by
	\begin{equation}\label{eq:eval_quasi}
	\lambda_i^\alpha={a} C_{11}^\alpha \pm \sqrt{{a}^2|C_{12}^\alpha|^2-{b}^2\big((C_{11}^\alpha)^2 - |C_{12}^\alpha|^2\big)},
	\end{equation}
	meaning the asymptotic exceptional point occurs when $b= b_0(\alpha),$ given by
	$${b}_0(\alpha) = \frac{{a} |C_{12}^\alpha|}{\sqrt{(C_{11}^\alpha)^2 - |C_{12}^\alpha|^2}}.$$
	The exceptional point now depends both on the geometry and on $\alpha$, and will therefore correspond to a point in the band structure. This is illustrated in \Cref{fig:bandexceptional}, which shows the band structure of a $\mathcal{PT}$-symmetric metascreen. The computations were performed using the multipole discretization as outlined in \cite[Appendix C]{ammari2017subwavelength}. Close to the origin of the Brillouin zone the system is always below the asymptotic exceptional point. For larger $\alpha$ and for large enough $b$, there will be a point $\alpha_0$ where $b = b_0(\alpha_0)$. For $\alpha$ above this point, the band structure of the system has a non-zero imaginary part and the two bands are complex-conjugated to leading order in $\delta$. {We remark that the error term is visibly much smaller in this case than it was for a system of two resonators} (\Cref{fig:exceptional}){, as discussed in }\Cref{rmk:asymdiff}.
	
	\begin{figure}[tbh] 
		\begin{center}
			\includegraphics[width=0.6\linewidth]{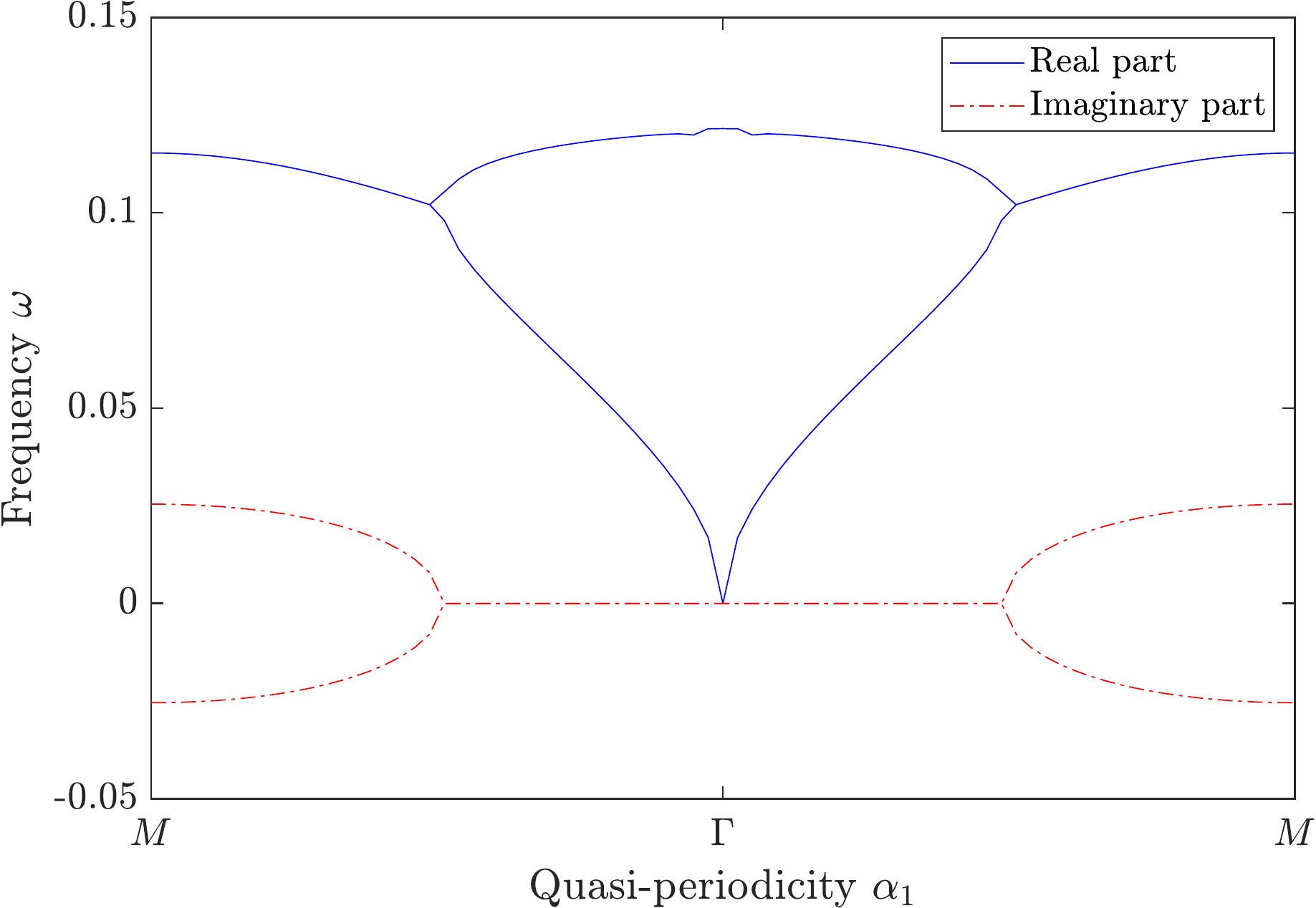}
		\end{center}
		\caption{Plot of the real parts (blue) and imaginary parts (red) of the band structure of the metascreen. The exceptional point is a point $(\alpha,\omega)$ at which the frequencies coincide to leading order. Here, we simulate a two-dimensional problem using circular resonators with period $L=1$, separation distance $0.5L$, radius $0.15L$ and material parameters $a = 2\times 10^{-4}, b = 1\times 10^{-4}$ and $v=1$. {These points correspond to $\delta \approx 2.2\cdot 10^{-4}$. Moreover, $\Gamma$ and $M$ indicate, respectively, the points $\alpha = 0$ and $\alpha = \pi/L$.}} \label{fig:bandexceptional}
	\end{figure}
	
	\subsection{Periodic Green's functions and capacitance matrix} \label{sec:C0}
	We will now study the layer potentials and capacitance coefficients when both $\omega$ and $\alpha$ approach zero. We consider $k\in \R$ and study the regime when $|\alpha| < k < \inf_{q\in \Lambda^*\setminus \{0\} }|\alpha+q|$, which is not encompassed by the analysis in \Cref{sec:band}. In this case, it was shown in \cite{ammari2017screen} that the quasiperiodic Green's function admits the spectral representation
	\begin{equation} \label{eq:specrep} G^{\alpha,k}(x) = \frac{e^{\iu\alpha\cdot(x_1,x_2)}e^{\iu k_3|x_3|}}{2\iu k_3L^2} - \sum_{q\in \Lambda^*\setminus \{0\}}\frac{e^{\iu(\alpha+q)\cdot(x_1,x_2)}e^{-\sqrt{|\alpha+q|^2 - k^2}|x_3|}}{2L^2\sqrt{|\alpha+q|^2 - k^2} },\end{equation}
	where $k_3 = \sqrt{k^2 - |\alpha|^2}$. The series in \eqref{eq:specrep} converges uniformly for $x$ in compact sets of $\R^3$, $x\neq 0$,  and $k \neq |\alpha+q|$ for all $q \in \Lambda^*$ (again, see \emph{e.g.} \cite{ammari2018mathematical}).% In the case $k < \inf_{q\in \Lambda^*}|\alpha+q|$ ("below the sound cone"), we similarly have
	%$$G^{\alpha,k}(x) = - \sum_{q\in \Lambda^*}\frac{e^{\iu(\alpha+q)\cdot(x_1,x_2)}e^{-\sqrt{|\alpha+q|^2 - k^2}|x_3|}}{2L^2\sqrt{|\alpha+q|^2 - k^2} }.	$$

	In the case when $k = \alpha = 0$, we call $G^{0,0}$ the \emph{periodic} Green's function and we have \cite{ammari2017screen}
	\begin{equation} \label{eq:specrep_0}
	G^{0,0}(x) = \frac{|x_3|}{2L^2} - \sum_{q\in \Lambda^*\setminus \{0\}}\frac{e^{\iu q \cdot(x_1,x_2)}e^{-|q||x_3|}}{2L^2|q|}.
	\end{equation}
	In the subsequent sections, we are interested in the case when the incident wave has a fixed direction of incidence and a frequency $\omega$ in the subwavelength regime. We therefore set  $k = \omega, \alpha = \omega \alpha_0$ and $w_3 = \sqrt{1-|\alpha_0|^2}$ for some $\alpha_0$ independent of $\omega$ (here, $\alpha_0$ represents the incident direction). When $\omega \rightarrow 0$, we have
	\begin{equation} \label{eq:Gexp}
	G^{\omega\alpha_0,\omega}(x) = \frac{1}{2\iu \omega w_3 L^2} + G^{0,0}(x) + \frac{\alpha_0 \cdot (x_1,x_2)}{2w_3L^2} + \omega G_1^{\alpha_0} (x) + O(\omega^2).
	\end{equation}
	Here, $G_1^{\alpha_0} $ is a function independent of $\omega$, which can be written \cite{ammari2017screen}
	$$
	G_1^{\alpha_0} (x) = \frac{\iu \left(w_3|x_3| + \alpha_0\cdot (x_1,x_2)\right)^2}{4w_3L^2} + \alpha_0 \cdot g_1(x),$$
	where $g_1(x)$ is a vector-valued function independent of $\alpha$ and $\omega$, satisfying 
	$$g_1(x_1,x_2,x_3) = g_1(x_1,x_2,-x_3), \qquad g_1(x_1,x_2,x_3) = - g_1(-x_1,-x_2,x_3).$$
	From \eqref{eq:Gexp} we in particular observe that the Green's function has a singularity of order $\omega^{-1}$. This fact makes the subsequent analysis qualitatively similar to the case of finite resonator systems in two dimensions, studied for example in \cite{bryn2019cochlea, ammari2018minnaert}. 
	
	We define the operators $\hat{\S}_D^{\alpha,k}: L^2(\p D) \rightarrow H^1(\p D)$ and $(\hat\K_D^{-\alpha,k})^*: L^2(\p D) \rightarrow L^2(\p D)$ as
	\begin{equation}
	\label{eq:Shat}\hat{\S}_D^{\alpha,k}[\varphi](x) = {\S}_D^{0,0}[\varphi](x) - \frac{\iu - \alpha\cdot(x_1,x_2)}{2 k_3 L^2} \int_{\p D}\varphi \dx \sigma - \int_{\p D}\frac{\alpha \cdot (y_1,y_2)}{2k_3L^2} \varphi(y)\dx \sigma (y),
	\end{equation}
	and
	$$(\hat\K_D^{-\alpha,k})^*[\varphi](x) = (\K_D^{0,0})^*[\varphi](x) +  \frac{\alpha \cdot (\nu_{x,1}, \nu_{x,2})}{2k_3L^2}\int_{\p D} \varphi \dx \sigma.$$
	Here, ${\S}_D^{0,0}$ and $(\K_D^{0,0})^*$ denote the \emph{periodic} single-layer potential and Neumann-Poincaré operators associated to $G^{0,0}$, while $\nu_x = (\nu_{x,1}, \nu_{x,2}, \nu_{x,3})$ denotes the outwards pointing normal of $D$ at $x$.
	Moreover, we define the operators $\S_1^{\alpha_0} : L^2(\p D) \rightarrow H^1(\p D)$ and  $(\K_{D,1}^{-\alpha_0})^*: L^2(\p D) \rightarrow L^2(\p D)$ as
	$$\S_1^{\alpha_0}[\phi](x) := \int_{\partial D} G_1^{\alpha_0} (x-y) \phi(y) \dx\sigma(y), \qquad (\K_{D,1}^{-\alpha_0})^*[\phi](x):= \int_{\p D} \frac{\p}{\p\nu_x} G_1^{\alpha_0}(x-y) \phi(y) \dx\sigma(y).$$
	Using \eqref{eq:Gexp} we can prove the following asymptotic expansions (following the arguments used to derive \cite[eq. (3.23)]{ammari2017screen})
	\begin{equation} \label{eq:exp_hat}
	\S_D^{\omega\alpha_0,\omega} = \hat{\S}_D^{\omega\alpha_0,\omega} + \omega\S_1^{\alpha_0} + O(\omega^2), \qquad (\K_D^{-\omega\alpha_0,\omega})^* = (\hat\K_D^{-\omega\alpha_0,\omega})^* +  \omega(\K_{D,1}^{-\alpha_0})^* + O(\omega^2),
	\end{equation}
	as $\omega \rightarrow 0$, where the error terms are with respect to corresponding operator norms. Similarly to \Cref{lem:ints}, we have the following lemma.
	\begin{lemma} \label{lem:intK}
		For any $\varphi\in L^2(\D)$ we have, for $i=1,2$,
	$$\int_{\D_i}\left(-\frac{1}{2}I+(\hat\K_D^{-\alpha,k})^*\right)[\varphi]\de\sigma=0, \qquad \int_{\D_i}(\K_{D,1}^{-\alpha_0})^*[\varphi] = \frac{\iu|D_1|}{2w_3L^2}\int_{\D}\varphi \de\sigma.$$
	\end{lemma}
	\begin{proof}
	For any $\varphi \in L^2(\p D)$ we have \cite{ammari2018mathematical}
	$$\int_{\D_i}\left(-\frac{1}{2}I+(\K_D^{0,0})^*\right)[\varphi]\de\sigma=0,$$
	and since $\int_{\D} \nu_x \dx \sigma(x) = 0$, the first equation follows. To prove the second equation, we use the first equation to conclude that, as $\omega\rightarrow 0$,
	\begin{equation}\label{eq:lemint1}
	\int_{\D_i}\left(-\frac{1}{2}I+(\K_D^{-\omega\alpha_0,\omega})^*\right)[\varphi]\de\sigma = \omega\int_{D_i}(\K_{D,1}^{-\alpha_0})^*[\varphi] + O(\omega^2).
	\end{equation}
	On the other hand, using the jump condition and integration by parts we have that
	\begin{align} \label{eq:lemint2}
	\int_{\D_i}\left(-\frac{1}{2}I+(\K_D^{-\omega\alpha_0,\omega})^*\right)[\varphi]\de\sigma &= \int_{D_i}\Delta\S_D^{\omega\alpha_0,\omega}[\varphi] \de x = -\omega^2\int_{D_i}\S_D^{\omega\alpha_0,\omega}[\varphi] \de x \nonumber \\
	&= \omega\frac{\iu|D_1|}{2w_3L^2}\int_{\D}\varphi \de\sigma + O(\omega^2). 
	\end{align}
	Since \eqref{eq:lemint1} and \eqref{eq:lemint2} hold for any small $\omega$, we obtain the second equation.
	\end{proof}
	The periodic single-layer potential could fail to be invertible and its kernel is described in the next lemma.
	\begin{lemma} \label{lem:ker}
		The dimension of 
		$\ker \S_D^{0,0}$ is at most one. Moreover, if $\varphi \in  \ker \S_D^{0,0}$ satisfies $\int_{\p D} \varphi \dx \sigma = 0,$ then $\varphi = 0$.
	\end{lemma}
\begin{proof}
	For small but non-zero $k$ we know from \Cref{lem:inv} that $\S_D^{0,k}$ is invertible, and therefore, by \eqref{eq:exp_hat}, $\hat\S_D^{0,k}$ is also invertible for small $k$. We can write $\S_D^{0,0}$ as
	$$\S_D^{0,0}[\varphi] = \hat{\S}_D^{0,k}[\varphi] + \frac{\iu }{2 k_3 L^2} \int_{\p D}\varphi \dx \sigma,$$
	or, in other words, $\S_D^{0,0}$ is a rank-1 perturbation of the invertible operator $\hat\S_D^{0,k}$. This shows that $\dim \ker \S_D^{0,0} \leq 1$ and, moreover, that any non-zero $\varphi  \in\ker \S_D^{0,0}$ satisfies $\int_{\p D} \varphi \dx \sigma \neq 0$.
	\end{proof}
\begin{lemma} \label{lem:phi=0}
	If $\S_D^{0,0}[\varphi] = K\Chi_{\p D}$ for some constant $K$ and some $\varphi \in L^2(\p D)$ satisfying $\int_{\p D} \varphi \dx \sigma = 0$, then $\varphi = 0$.
\end{lemma}
\begin{proof}
	For $x\in \R^3\setminus \C$, define $V(x) := \S_D^{0,0}[\varphi](x)$. Then $V$ solves the following differential problem,
	\begin{equation} \label{eq:V}
	\left\{
	\begin{array} {ll}
	\ds \Delta V  = 0 & \text{in } \R^3\setminus \C, \\[0.3em]
	\nm
	\ds  V|_{+}  = K  & \text{on } \partial \C, \\[0.3em]
	\nm
	\ds  V(x+(m_1,m_2,0))  = V(x)  & \text{for all } (m_1,m_2) \in \Lambda, \\
	\nm
	\ds V(x) \rightarrow \pm V_\infty & \text{as }  x_3 \rightarrow \pm\infty,
	\end{array}
	\right.
	\end{equation}
	for some constant $V_\infty$. Moreover, using the jump relations and integration by parts, we have that 
	$$\int_{\p D} \varphi \dx \sigma = K\int_{Y \setminus D} |\nabla V|^2 \dx x=0.$$
	If $K\neq 0$, it follows from \eqref{eq:V} that $\int_{Y \setminus D} |\nabla V|^2 \dx x\neq0$. In other words we must have $K=0$ and, since $\int_{\p D} \varphi \dx \sigma = 0$, it follows from \Cref{lem:ker} that $\varphi = 0$.
\end{proof}
Let $L^2_0(\p D)$ be the mean-zero space defined as $$L^2_0(\p D) = \left\{f \in L^2(\p D) \ \left| \ \int_{\p D} f \dx \sigma = 0 \right. \right\}.$$
By \Cref{lem:ker} and \Cref{lem:phi=0}, $\S_D^{0,0}$ is invertible from $L^2_0(\p D)$ onto its image, which does not contain the constant functions.

We will now define the analogous capacitance coefficients in the periodic setting. We begin with the following lemma.
	\begin{lemma} \label{lem:holo}
		For any $\alpha_0 \in Y^*$ with $|\alpha_0| < 1$,  $\left(\hat{\S}_D^{\omega\alpha_0,\omega}\right)^{-1}$ is a holomorphic operator-valued function of $\omega$ in a neighbourhood of $\omega = 0$.
	\end{lemma}
\begin{proof}
	We know that $\hat{\S}_D^{\omega\alpha_0,\omega}$ is a meromorphic operator-valued function of $\omega$ with a pole at $\omega = 0$. From \cite[Corollary 1.10]{ammari2018mathematical}, we find that $\left(\hat{\S}_D^{\omega\alpha_0,\omega}\right)^{-1}$ is also meromorphic for $\omega$ in a neighbourhood of $0$. It remains to show that the principal part vanishes.
	
	To reach a contradiction, we assume that $\left(\hat{\S}_D^{\omega\alpha_0,\omega}\right)^{-1}$ is singular as $\omega \rightarrow 0$, which means that there is some $\phi$, depending on $\omega$, such that $\|\phi\|_{L^2(\p D)} = O(1)$ while $\|\hat{\S}_D^{\omega\alpha_0,\omega}[\phi]\|_{H^1(\p D)} = O(\omega)$. We can rewrite $\phi$ as $\phi = \phi_0 + \phi_1$ where $\phi_0$ is non-zero and independent of $\omega$, while $\phi_1 = O(\omega)$. Then the singular part of $\hat{\S}_D^{\alpha,k}$ must vanish on $\phi_0$, \ie{},
	$$\int_{\p D}\phi_0 \dx \sigma = 0.$$
	Substituting into \eqref{eq:Shat} we find that, for some constant $K$, we have $\S_D^{0,0}[\phi_0] = K\Chi_{\p D}.$
	It then follows from \Cref{lem:phi=0} that $\phi_0 = 0$, which contradicts the fact that $\| \phi \| = O(1)$.
	\end{proof}
	We can now define the periodic capacitance coefficients $C^{0}_{ij}$. For $\alpha \in Y^*$, we let
	$$\psi_i^{\alpha,\omega} = \left(\hat{\S}_D^{\alpha,\omega}\right)^{-1}[\Chi_{\p D_i}]. $$
	Then, if $\alpha = \omega\alpha_0$ for some fixed $\alpha_0$ with $|\alpha_0| < 1$, we have the following expansion from \Cref{lem:holo},
	\begin{equation} \label{eq:psi0}
	\psi_i^{\omega\alpha_0,\omega} = \psi_i^{0} + \omega\hat{\psi}_i^{1,\alpha_0} + O(\omega^2),
	\end{equation}
	as $\omega\to 0$, for some $\psi_i^{0}, \hat{\psi}_i^{1,\alpha_0} \in L^2(\p D)$ independent of $\omega$. We then define		
	\begin{equation}
	C_{ij}^{0}=-\int_{\D_i} \psi_j^{0}\de\sigma.
	\end{equation}
	We call the matrix $C^{0}  = (C_{ij}^{0})$ the \emph{periodic capacitance matrix} (not to be confused with the \emph{quasi}periodic capacitance matrix, studied in \Cref{sec:band}). The periodic capacitance matrix might \emph{a priori} depend on $\alpha_0$, but we will later see that, under an extra symmetry condition, $\psi_j^{0}$ and $C^{0}$ are independent of $\alpha_0$.
	\begin{lemma} \label{lem:cap0}
		The periodic capacitance matrix $C^{0}$ is a real matrix given by		
		\begin{equation*}
		C^{0} = C_{11}^{0}\begin{pmatrix} 1 & -1 \\ -1 & 1\end{pmatrix}.
		\end{equation*}
		\end{lemma}
\begin{proof}
	Since $\hat \S_D^{\omega \alpha_0,\omega}[\psi_i^{\omega\alpha_0,\omega}]$ is bounded as $\omega \to 0$, the singular part of $\hat \S_D^{\omega \alpha_0,\omega}$ must vanish on $ \psi_i^{{0}}$, or in other words
	$$\int_{\p D} \psi_i^{{0}}  \dx \sigma = 0.$$
	From this it follows that $C_{i1}^{0} = - C_{i2}^{0}$ for $i=1,2$. From the condition $\hat \S_D^{\omega \alpha_0,\omega}[\psi_i^{\omega\alpha_0,\omega}] = \Chi_{\p D_i}$, we have that 
	\begin{align} \label{eq:leadingorder}
	\Chi_{\p D_i} &= \S_D^{0,0}[\psi_i^{0}] + \frac{1}{2\iu w_3 L^2}\int_{\p D}\hat{\psi}_i^{1,\alpha_0} \dx \sigma - \int_{\p D}\frac{\alpha_0 \cdot (y_1,y_2)}{2w_3L^2} \psi_i^{0}(y)\dx \sigma (y),
	\end{align}
	or, in other words, that $\S_D^{0,0}[\psi_i^{0}] = \Chi_{\p D_i} + K\Chi_{\p D}$
	for some constant $K$. Summing over the resonators, we find that $\S_D^{0,0}[\psi_1^{0} + \psi_2^{0}] = \widetilde K\Chi_{\p D}$. By \Cref{lem:phi=0} we find that $\widetilde K = 0$ and that 
	\begin{equation} \label{eq:psi1,2}
		\psi_1^{0} = - \psi_2^{0}.
	\end{equation}
	It follows that $C_{1j}^{0} = -C_{2j}^{0}$ for $j=1,2$, which proves the expression of $C^{0}$. It remains to prove that $C_{11}^{0}$ is real. Taking the complex conjugate of \eqref{eq:leadingorder} we find that $\S_D^{0,0}[\psi_i^{0} - \overline{\psi_i^{0}}] = K\Chi_{\D}$ for some new constant $K$. From \Cref{lem:phi=0} we find that $\psi_i^{0} = \overline{\psi_i^{0}}$, and hence $ C_{11}^{0} = \overline{ C_{11}^{0}}$.
\end{proof}

%At $\alpha_0=0$, we can show that $\hat{\psi}_1^{1,0} = \hat{\psi}_2^{1,0}$. Moreover, if $\S_D^{0,0}[\varphi] = K\Chi_{\p D}$ for some $\varphi$ and constant $K$, it follows that, for some other constant $\tilde{K}$,
%$$\varphi = \tilde{K}\hat{\psi}_1^{1,0}.$$
%If $\S_D^{0,0}$ has a nontrivial kernel we must have $K=0$ and the kernel is spanned by $\hat{\psi}_1^{1,0}$. 

\begin{lemma} \label{lem:psi_err}
	As $\omega \to 0$, we have 
	$$\left({\S}_D^{\omega\alpha_0,\omega}\right)^{-1}[\Chi_{\p D_j}] = \psi_j^{{0}} + \omega \psi_j^{1,\alpha_0} +  O(\omega^2),$$
	where 
	$$ \psi_j^{1,\alpha_0} = \hat{\psi}_j^{1,\alpha_0} - \left(\hat\S_D^{\omega\alpha_0,\omega}\right)^{-1}\S_1^{\alpha_0} \psi_j^{0}.$$
\end{lemma}
\begin{proof}
	From \eqref{eq:exp_hat}, and using the Neumann series, we have
\begin{align*}
\left({\S}_D^{\omega\alpha_0,\omega}\right)^{-1}[\Chi_{\p D_j}] &= \left(\hat{\S}_D^{\omega\alpha_0,\omega}\right)^{-1}[\Chi_{\p D_j}] - \omega\left(\hat{\S}_D^{\omega\alpha_0,\omega}\right)^{-1}\S_1^{\alpha_0}\left(\hat{\S}_D^{\omega\alpha_0,\omega}\right)^{-1}[\Chi_{\p D_j}] + O(\omega^2), \\
&= \psi_j^{{0}} + \omega \psi_j^{1,\alpha_0} +  O(\omega^2),
\end{align*}
which proves the claim.
\end{proof}

	Analogously to before, we define the weighted periodic capacitance matrix as 
	\begin{equation*}
	C^{v,{0}}=V C^{0}.
	\end{equation*}
	From \Cref{lem:cap0}, we find that the eigenvalues $\lambda_1^{0}, \lambda_2^{0}$ and corresponding eigenvectors $\v_1^{0}, \v_2 ^{0}$ of $C^{v,{0}}$ are given by
	$$\lambda_1^{0} = 0,\quad \lambda_2^{0} = 2aC_{11}^{0}, \qquad \v_1^{0} = \begin{pmatrix} 1 \\ 1 \end{pmatrix}, \quad \v_2^{0} = \begin{pmatrix} -(a + \iu b) \\ a-\iu b \end{pmatrix}.$$
	As we shall see, the weighted periodic capacitance matrix asymptotically describes the resonant frequencies and the scattered field to leading order. In addition, we will need to consider two sources of higher-order effects. Firstly, we define the vector-valued coefficients $\mathbf{c}_i$ as
	\begin{equation}
	\mathbf{c}_i=\int_{\D} y\psi_i^{{0}}(y)\de\sigma(y), \quad i = 1,2.
	\end{equation}	
	From \eqref{eq:psi1,2} we have that $\mathbf{c}_1 = -\mathbf{c}_2$. Secondly, we define the matrix $C^{1,\alpha_0} = (C_{ij}^{1,\alpha_0})$ as 
	$$C_{ij}^{1,\alpha_0} = -\int_{\D_i} \psi_j^{1,\alpha_0} \dx \sigma,$$
	for $i,j = 1,2$. Corresponding weighted matrix $C^{v,1,\alpha_0} = (C_{ij}^{v,1,\alpha_0})$ is defined as
	$$C^{v,1,\alpha_0} = VC^{1,\alpha_0}.$$		
	The next lemma describes some of the structure of $C^{1,\alpha_0}$.
	\begin{lemma} \label{lem:C1}
		We have 
		$$ \int_{\p D}{\psi}_j^{1,\alpha_0} \dx \sigma = \int_{\p D}\hat{\psi}_j^{1,\alpha_0} \dx \sigma + O(\omega) \quad \text{and} \quad 
		\int_{\p D}{\psi}_1^{1,\alpha_0} + {\psi}_2^{1,\alpha_0} \dx \sigma = 2\iu w_3 L^2.
		$$
	\end{lemma}
\begin{proof}
	Using the fact that the $L^2(\p D)$-dual of $\hat{\S}_D^{\alpha,k}$ is $\overline{\hat{\S}_D^{-\alpha,k}}$, we have
	$$\int_{\p D} \left(\hat\S_D^{\omega\alpha_0,\omega}\right)^{-1}\S_1^{\alpha_0} [\psi_j^{0}] = \int_{\p D} \S_1^{\alpha_0} [\psi_j^{0}] \left(\hat\S_D^{-\omega\alpha_0,\omega}\right)^{-1}[\Chi_{\D}] \dx \sigma = O(\omega).$$
Moreover, since $\psi_1^{0} = -\psi_2^{0}$ we have
$$
\int_{\p D}{\psi}_1^{1,\alpha_0} + {\psi}_2^{1,\alpha_0} \dx \sigma = \int_{\p D}\hat{\psi}_1^{1,\alpha_0} + \hat{\psi}_2^{1,\alpha_0} \dx \sigma =2\iu w_3 L^2,
$$
where the last step follows from \eqref{eq:leadingorder} together with \eqref{eq:psi1,2}. This proves the claim.
\end{proof}

\begin{remark}
	It is straightforward to generalise \Cref{lem:cap0} to a general number $N$ of resonators inside the unit cell. The weighted periodic capacitance matrix $C^{v,0}$ will always have one vanishing eigenvalue. This corresponds to the well-known fact that the first band function $\omega_1^\alpha$ satisfies $\omega_1^0 = 0$ corresponding to monopole modes $\v_1^0 =  (1, ..., 1)^\mathrm{T} \in \R^N$. The other eigenvalues of $C^{v,0}$ describe the values of the other band functions $\omega_2^\alpha, \omega_3^\alpha, ..., \omega_N^\alpha$ around $\alpha = 0$.
\end{remark}
	\subsection{Plane wave scattering problem} \label{sec:scatter_asymp}
	We assume that the incident field $u^{\mathrm{in}}$ is a plane wave with frequency $\omega \in \R$ and wave vector $\k = \begin{pmatrix}k_1 & k_2 &k_3\end{pmatrix}^{\mathrm{T}}$. Again, superscript $\mathrm{T}$ denotes the transpose operator. In other words,
	$$u^{\mathrm{in}}(x) =
	e^{\iu\k\cdot x}, \qquad |\k| = k = \frac{\omega}{v},$$
	where $|\k|$ denotes the Euclidean norm of $\k$. For simplicity, we assume that the units are chosen such that $v=1$. We will consider the subwavelength regime, \ie{} when $\delta \rightarrow 0$ and $\omega = O(\delta^{1/2})$. In this limit, we assume that the incident direction $\w$ of $\k$ is fixed, \ie{} that $\k$ scales as 
	$$\k = \omega\w, \quad \w = \begin{pmatrix}
	w_1 \\ w_2 \\ s w_3
	\end{pmatrix}, \quad w_3 > 0, \quad  s = \pm 1,$$
	where $\w$ is independent of $\omega$. We define
	$$\alpha = \begin{pmatrix}k_1 \\ k_2 \end{pmatrix} = \omega \alpha_0 \in Y^*.$$
	We define the functions $S_j^{\alpha,\omega}$, for $j=1,2$, as	
	$$S_j^{\alpha,\omega}(x) = \begin{cases}
	\S_{D}^{\alpha,k}[\psi_j^{0} + \omega \psi_j^{1,\alpha_0}](x), & x\in \R^3 \setminus \overline\C,\\
	\S_{D}^{\alpha,k_i}[\psi_j^{0} + \omega \psi_j^{1,\alpha_0}](x), & x\in \C_i, \ i=1,2.
	\end{cases}
	$$
	\begin{prop} \label{lem:modal_quasi}
		Let $\lambda_2^{0}, \v_2^{0}$ be the second eigenpair of $C^{v,{0}}$, and let $\lambda = \omega^2|D_1|.$ Assume that $\Im(d^\mathrm{T} C^{v,1,\alpha_0} \v_2^{0}) \neq 0$, where $d =\left( \begin{smallmatrix}1 \\ -1
		\end{smallmatrix}\right)$. Then, for $\omega \in \R$ in the subwavelength regime such that $\lambda = \lambda_2^{0} + \lambda^*$, where $\lambda^* = O(\omega^3)$, the solution to the scattering problem \eqref{eq:scattering_quasi} can be written as 
		$$u - u^{\mathrm{in}}=-(a+\iu b)\mu S_1^{\alpha,\omega} + (a-\iu b)\mu S_2^{\alpha,\omega} - \S_D^{\alpha,k}\left(\S_D^{\alpha,k}\right)^{-1}[u^{\mathrm{in}}] + O(\omega^2),$$
		where $\mu$ is given by 
		$$\mu = \frac{ d^\mathrm{T}p } {d^\mathrm{T} \left(\omega C^{v,1,\alpha_0} - \lambda^*I\right) \v_2^{0}} + O(\omega), \qquad p= -\begin{pmatrix} {v_1^2\delta_1}\int_{\D_1}\left({\S}_D^{\alpha,k}\right)^{-1}[u^{\mathrm{in}}]\de\sigma \\[0.3em]
		{v_2^2\delta_2}\int_{\D_2}\left({\S}_D^{\alpha,k}\right)^{-1}[u^{\mathrm{in}}]\de\sigma \end{pmatrix}.$$
		Here, the error terms are uniform with respect to $\lambda^*$ in a neighbourhood of $0$.
	\end{prop}
\begin{proof}
The solutions to \eqref{eq:scattering_quasi} can be represented as
\begin{equation} \label{eq:rep_quasi}
u = \begin{cases}
u^{\mathrm{in}}(x)+\S_{D}^{\alpha,k}[\psi](x), & x\in \R^3 \setminus \overline{\C},\\
\S_{D}^{k_i}[\phi](x), & x\in \C_i, \ i=1,2,
\end{cases}
\end{equation} 
for some surface densities $(\phi,\psi)\in L^2(\D)\times L^2(\D)$, which must be chosen so that $u$ satisfies the transmission conditions across $\D$.
Using the jump conditions \eqref{eq:jump1_quasi} and \eqref{eq:jump2_quasi}, we see that this implies that the layer densities $\phi$ and $\psi$ satisfies
\begin{align}
\S_{D}^{k_i}[\phi]-\S_{D}^{\alpha,k}[\psi]=u^{\mathrm{in}} \quad \text{on}\ \D_i, \label{eq:quasi_full1} \\
\left(-\frac{1}{2}I+\K_{D}^{k_i,*}\right)[\phi]-\delta_i\left(\frac{1}{2}I+(\K_{D}^{-\alpha,k})^*\right)[\psi]=\delta_i \ddp{u^{\mathrm{in}}}{\nu} \quad \text{on} \ \D_i, \label{eq:quasi_full2}
\end{align}
for $i=1,2$.
Using the asymptotic expansions \eqref{eq:exp_hat} and \eqref{eq:exp_K} we have from \eqref{eq:quasi_full2} that, on $\D_i$,
\begin{gather} 
\left(-\frac{1}{2}I+\K_D^*+\frac{\omega^2}{v_i^2}\K_{D,2} +\frac{\omega^ 3}{v_i^3}\K_{D,3}\right)[\phi]-\delta_i\left(\frac{1}{2}I+(\hat\K_D^{-\alpha,k})^* + \omega(\K_{D,1}^{-\alpha,k})^*\right)[\psi]=O(\delta\omega^2+\omega^4). \label{eq:quasi2}
%\end{split}
\end{gather}
Integrating over $\D_i$, using \Cref{lem:ints,lem:intK} along with \eqref{eq:properties}, gives us that
\begin{equation} \label{eq:int}
-\frac{\omega^2}{v_i^2}\int_{D_i}\S_D[\phi]\de x + \frac{\omega^3}{v_i^3}\frac{\iu|D_i|}{4\pi}\int_{\D}\phi\de\sigma - \delta_i\int_{\D_i}\psi\de\sigma - \delta_i\frac{\iu\omega|D_1|}{2w_3L^2} \int_{\D}\psi\de\sigma= O(\delta\omega^2+\omega^4).
\end{equation}

At leading order, \eqref{eq:quasi2} says that $\left(-\frac{1}{2}I+\K_D^{*}\right)[\phi]=0$ so, in light of the fact that $\psi_1$ and $\psi_2$ form a basis for $\ker\left(-\frac{1}{2}I+\K_D^{*}\right)$, $\phi$ can be written as
\begin{equation} \label{eq:psi_basis_quasi}
\phi=q_1\psi_1+q_2\psi_2+O(\omega^2+\delta),
\end{equation}
for constants $q_1,q_2=O(1)$. Using \eqref{eq:exp_S}, we can expand $\S_D^{k_i}[\psi_j]$ as 
$$\S_D^{k_i}[\psi_j] = \Chi_{\D_i} - \frac{\omega\mathrm{Cap}_D}{8\pi \iu v_i}\Chi_{\D} + O(\omega^2),$$
where $\mathrm{Cap}_D = 2(C_{11} + C_{12})$. From \eqref{eq:quasi_full1}, we then find that
$$\S_D^{\alpha,k}[\psi] = \Chi_{\D_1}\left( q_1 - (q_1+q_2) \frac{\omega\mathrm{Cap}_D}{8\pi \iu v_1}\right) + \Chi_{\D_2}\left( q_2 - (q_1+q_2) \frac{\omega\mathrm{Cap}_D}{8\pi \iu v_2}\right)+O(\omega^2+\delta),$$
and then from \Cref{lem:psi_err} that 
\begin{multline} \label{eq:psi_basis_exp}
\psi = q_1\left(\psi_1^{0}+ \omega\psi_1^{1,\alpha_0} -  \frac{\omega\mathrm{Cap}_D}{8\pi \iu }\left(\frac{\psi_1^{0}}{v_1} + \frac{\psi_2^{0}}{v_2} \right)\right) \\
+ q_2 \left(\psi_2^{0}+ \omega\psi_2^{1,\alpha_0} - \frac{\mathrm{\omega Cap}_D}{8\pi \iu }\left(\frac{\psi_1^{0}}{v_1} + \frac{\psi_2^{0}}{v_2}\right) \right) -\left({\S}_D^{\alpha,k}\right)^{-1}[u^{\mathrm{in}}] + O(\omega^2 + \delta).
\end{multline}
Substituting \eqref{eq:psi_basis_quasi} and \eqref{eq:psi_basis_exp} into \eqref{eq:int}, and using the fact that $\int_{\D}\left({\S}_D^{\alpha,k}\right)^{-1}[u^{\mathrm{in}}]\de\sigma = O(\omega)$, we reach the problem 
\begin{equation}\label{eq:eval_C_quasi}
\left({C}^{v,0}-\lambda I + E \right)\begin{pmatrix}q_1\\q_2\end{pmatrix}
= -\begin{pmatrix} {v_1^2\delta_1}\int_{\D_1}\left({\S}_D^{\alpha,k}\right)^{-1}[u^{\mathrm{in}}]\de\sigma \\[0.3em]
{v_2^2\delta_2}\int_{\D_2}\left({\S}_D^{\alpha,k}\right)^{-1}[u^{\mathrm{in}}]\de\sigma \end{pmatrix}+O(\delta \omega^2+\omega^4),
\end{equation}
where $\lambda = \omega^2|D_1|$, while $E= (E_{i,j}) = O(\delta\omega + \omega^3), \ i,j = 1,2$ is the matrix given by
$$E_{i,j} = \frac{\omega\mathrm{Cap}_D}{8\pi\iu}\left( \frac{|D_1|\omega^2}{v_i} - \frac{C^{v,0}_{i1}}{v_1} - \frac{C^{v,0}_{i2}}{v_2} \right)- v_i^2\delta_i\omega\int_{\D_i} \psi_j^{1,\alpha_0} \dx \sigma.$$
We write $q = \left(\begin{smallmatrix} q_1 \\ q_2 \end{smallmatrix}\right)$ and denote the right-hand side of \eqref{eq:eval_C_quasi} by $p$. Recall that we are working in the subwavelength regime $\omega = O(\delta^{1/2})$. Assuming $\lambda = \lambda_2^{0} + \lambda^*$, where $\lambda^* = O(\omega^3)$, we can rewrite \eqref{eq:eval_C_quasi} into
\begin{equation}\label{eq:eval_C_quasi_short}
\left({C}^{v,0}-\lambda_2^{0} I + E - \lambda^*I \right)q
= p.
\end{equation}
Using the second eigenvector $\v_2^{0}$ of ${C}^{v,0}$, we can find a constant $\mu$ such that
$$q = \mu \v_2^{0} + q_0,$$
for some $q_0$ satisfying $\v_2^{0} \cdot q_0 = 0$.

Next, we compute $p$. Since $u^{\mathrm{in}}$ is a plane wave with wave vector $\k = \omega\w$, we have
$$ u^{\mathrm{in}}(x) = 1 + \omega \iu \w\cdot x + O(\omega^2), \quad x \in \p D.$$
Using duality, we have that 
%\begin{align} \label{eq:q}
$$\int_{\D_i}\left({\S}_D^{\alpha,k}\right)^{-1}[u^{\mathrm{in}}]\de\sigma = \int_{\D}u^{\mathrm{in}}\left({\S}_D^{-\alpha,k}\right)^{-1}[\Chi_{\p D_i}]\de\sigma \nonumber = O(\omega).$$ We conclude that $p=O(\omega^3)$. Therefore \eqref{eq:eval_C_quasi_short} shows that $q_0=O(\omega)$. We are now able to compute $\mu$. Letting $d=\left(\begin{smallmatrix} 1 \\ -1\end{smallmatrix}\right) $, it is straightforward to compute 
$$\mu d^\mathrm{T} (E-\lambda^*I) \v_2^{0}  = d^\mathrm{T} p + O(\omega^4).$$
 We can simplify
$$E\v_2^{0} = - \frac{\omega b\mathrm{Cap}_DC_{11}^{0}}{4\pi}\left(\frac{a-\iu b}{v_1} + \frac{a+\iu b}{v_2} \right)\begin{pmatrix} 1 \\  1\end{pmatrix} + \omega C^{v,1,\alpha_0}\v_2^{0}.$$
Then 
$$d^\mathrm{T}(E-\lambda^*I) \v_2^{0} = d^\mathrm{T} \left(\omega C^{1,v,\alpha_0} - \lambda^*I\right) \v_2^{0}.$$
From the assumption $\Im(d^\mathrm{T}C^{1,v,\alpha_0}\v_2^0) \neq 0$, and since $\lambda^*$ is real, we find that $\left|d^\mathrm{T} \left(\omega C^{1,v,\alpha_0} - \lambda^*I\right) \v_2^{0}\right|> \omega^3K > 0$ for some constant $K$, for all $\lambda^*$ in a neighbourhood of $0$. We then have 
$$
\mu = \frac{d^\mathrm{T} p} {d^\mathrm{T} \left(\omega C^{v,1,\alpha_0} - \lambda^*I\right) \v_2^{0}} + O(\omega),$$
uniformly for $\lambda^*$ in a neighbourhood of $0$. Then, combining \eqref{eq:rep_quasi} and \eqref{eq:psi_basis_exp}, we find that for $x\in \R^3$,
\begin{align*}
u(x) - u^{\mathrm{in}}(x) &= -\mu (a + \iu b)S_1^{\alpha,\omega}(x) + \mu (a-\iu b)S_2^{\alpha,\omega}(x) - \S_D^{\alpha,k}\left(\S_D^{\alpha,k}\right)^{-1}[u^{\mathrm{in}}](x) +  O(\omega^2),
\end{align*}
(we emphasise that there is no cancellation in the last term for $x \notin \D$). This proves the claim. 
\end{proof}

\begin{remark} \label{rmk:first}
	If $\omega$ is instead close to the first resonant frequency $\omega_1^0 = 0$, the solution $q$ to \eqref{eq:eval_C_quasi_short} will be approximated by the first eigenvector $\v_1^{0}$. Consequently, it can be shown that $\psi$ vanishes to high order. In other words, the incoming wave is largely unaffected by the metascreen and the scattered field is small.
\end{remark}
	\subsection{Unidirectional reflection and extraordinary transmission}
	In this section, we prove that there is a frequency such that the metascreen's reflection coefficient is {asymptotically close to zero} when the incident wave is from one side and non-zero when the incident wave is from the other side of the screen. We will also demonstrate the occurrence of extraordinary transmission. The main results are stated in \Cref{thm:unidir}.
	
	We begin by studying the radiative behaviour of the basis functions $S_1^{\alpha,\omega}$ and $S_2^{\alpha,\omega}$, in terms of which the scattered field is expressed. The quasiperiodic radiation condition implies that the single layer potential behaves as a superposition of outgoing plane waves as $|x_3| \rightarrow \infty$. Throughout this section, we will use $\sim$ to denote equality up to exponentially decaying factors, \ie{} for functions $f,g\in C(\R)$ we have $f(x) \sim g(x), x\rightarrow \infty$ if and only if
	$$|f(x) - g(x)| = O(e^{-Kx}) \text{ as } x\rightarrow \infty,$$
	for some constant $K >0$. The following result describes the radiative behaviour of the single layer potential in the case of a single propagating mode, and is a direct consequence of the expansion of the Green's function in \eqref{eq:specrep}.
	\begin{prop} \label{prop:radiation}
		Assume that $|\alpha| < k < \inf_{q\in \Lambda^*\setminus \{0\} }|\alpha+q|$. Then, as $|x_3|\rightarrow \infty$, the quasiperiodic single layer potential satisfies 
		$$
		\S_D^{\alpha,k}[\phi] \sim \begin{cases}\ds \frac{e^{\iu\k_+\cdot x}}{2\iu k_3L^2}\int_{\D} e^{-\iu\k_+\cdot y}\phi(y)\dx \sigma(y), \quad &x_3\rightarrow \infty, \\[2em]
		\ds \frac{e^{\iu\k_-\cdot x}}{2\iu k_3L^2}\int_{\D} e^{-\iu\k_-\cdot y}\phi(y)\dx \sigma(y), \quad &x_3\rightarrow -\infty. 
		\end{cases}
		$$
		Here, $k_3 = \sqrt{k^2 - |\alpha|^2}$ while $\k_+ = (\alpha,k_3)^\mathrm{T}$ and $\k_- = (\alpha,-k_3)^\mathrm{T}.$
	\end{prop}
	We define the coefficients 
	$$R_{j,\pm} = \frac{1}{2 \iu k_3L^2}\int_{\D} e^{-\iu \k_\pm\cdot y}\left(\psi_j^{0}(y) + \omega\psi_j^{1,\alpha_0}(y)\right)\dx \sigma(y), \quad j = 1,2.$$
	By \Cref{prop:radiation}, the basis functions $S_1^{\omega, \alpha}, S_2^{\omega, \alpha}$ for the scattered field satisfies the radiative behaviour 
	\begin{equation}\label{eq_R}
	S_j^{\alpha, \omega} \sim R_{j,\pm}e^{\iu\k_\pm \cdot x}, \quad x_3 \rightarrow \pm \infty.
	\end{equation}

	\subsubsection{Scattering matrix and unidirectional reflectionless transmission}
	Recall that we are considering the limit when $\delta\rightarrow 0$ and supposing that $\omega = O(\sqrt{\delta})$. The condition  $|\alpha| < k < \inf_{l\in \mathbb{Z}^2\setminus \{0\} }|2\pi l L-\alpha|$ will be satisfied for small enough $\omega$, so the scattered wave will behave as a single plane wave as $|x_3| \rightarrow \infty$.
	If the incident field is given by
	$$u^{\mathrm{in}}(x) = c_1 e^{\iu \k_-\cdot x} + c_2 e^{\iu\k_+\cdot x},$$
	the total field will behave as 
	\begin{equation}\label{eq:utot}
	u \sim \begin{cases}c_1 e^{\iu\k_-\cdot x} + d_1 e^{\iu\k_+\cdot x}, & x_3\rightarrow \infty, \\ c_2 e^{\iu\k_+\cdot x} + d_2 e^{\iu\k_-\cdot x}, & x_3\rightarrow -\infty,\end{cases}
	\end{equation}
	where 
	\begin{equation} \label{eq:S_def}
	\begin{pmatrix} d_1 \\ d_2 \end{pmatrix} =  S\begin{pmatrix} c_1 \\ c_2 \end{pmatrix}, \qquad S = \begin{pmatrix} r_+ & t_- \\ t_+ & r_- \end{pmatrix}.
	\end{equation}
	$S$ is known as the scattering matrix.
	The reflection and transmission coefficients $r_+, t_+$ are the coefficients of the outgoing part of the field in the case $u^{\mathrm{in}}(x) =  e^{\iu\k_-\cdot x}$, \ie{} when the incident field is a plane wave from the positive $x_3$ direction (and reversely for $r_-, t_-$). Next, we will compute the scattering matrix in the asymptotic limit specified in \Cref{sec:scatter_asymp}.
	
	For simplicity, we set $u^{\mathrm{in}} = e^{\iu \k\cdot x}$ with $\k=\k_+$ or $\k=\k_-$, and then use linearity to obtain the full scattering matrix. From \Cref{lem:modal_quasi}, we know that the scattered field is given by
	\begin{equation}\label{eq:scattered}
	u - u^{\mathrm{in}} = -(a+\iu b)\mu S_1^{\alpha,\omega} + (a-\iu b)\mu S_2^{\alpha,\omega} - \S_D^{\alpha,k}\left(\S_D^{\alpha,k}\right)^{-1}[u^{\mathrm{in}}] + O(\omega^2).
	\end{equation}
	As $\omega \rightarrow 0$, we have the following asymptotic behaviour of
	\begin{align*}
	R_{j,\pm} &= \frac{1}{2\iu k_3L^2}\int_{\D}\psi_j^{0}(y)\dx \sigma(y) - \frac{1}{2 \iu k_3L^2}\int_{\D} \iu \k_\pm\cdot y\psi_j^{0}(y)\dx \sigma(y) + \frac{\omega}{2 \iu k_3L^2}\int_{\D} \psi_j^{1,\alpha_0}\dx \sigma +  O(\omega) \\
	&= - \frac{\k_\pm \cdot \mathbf{c}_j}{2 k_3L^2} + \frac{1}{2 \iu w_3L^2}\int_{\D} \psi_j^{1,\alpha_0}\dx \sigma  + O(\omega).
	\end{align*}
	Moreover,
	\begin{align*}
	\frac{1}{2\iu k_3L^2}\int_{\D} e^{-\iu\k_\pm\cdot y}\left(\S_D^{\alpha,\omega}\right)^{-1}[u^{\mathrm{in}}]\dx \sigma(y) = \frac{\omega}{2\iu k_3L^2}\int_{\D}\left( \psi_1^{1,\alpha_0} + \psi_2^{1,\alpha_0}\right) \dx \sigma + O(\omega) = 1 + O(\omega).
	\end{align*}
	Therefore, from \Cref{prop:radiation}, \eqref{eq_R} and \eqref{eq:scattered}, the scattered field satisfies
	\begin{align} \label{eq:uscattered}
	u-u^{\mathrm{in}} &\sim \left(\frac{\mu a \k_\pm \cdot \mathbf{c}_1}{ k_3L^2} - \iu \mu b - 1 + O(\omega)\right)e^{\iu\k_\pm \cdot x}, \\
	&=:\big(G_ {s\pm}(\lambda^*) + O(\omega)\big)e^{\iu\k_\pm \cdot x},
	\end{align}
	as $x_3\rightarrow \pm \infty.$ Here $s$ denotes the sign of the third component of $\k$ (recall that $\mu$ depends on $\k$). From \eqref{eq:utot} and \eqref{eq:uscattered}, it follows that the scattering matrix, defined in \eqref{eq:S_def}, can be written as
	\begin{equation}\label{eq:S}
	S =  \begin{pmatrix} G_ {-+}(\lambda^*) & 1+ G_ {++}(\lambda^*)\\ 1+ G_ {--}(\lambda^*)  & G_ {+-}(\lambda^*)\end{pmatrix} + O(\omega).
	\end{equation}
	Up to this point, the only assumption we have made on the resonators' geometry is that they are symmetric under the parity operator $\P$. In order to simplify the above expressions, we will additionally assume that the resonators have an in-plane parity symmetry $\P_2$, \ie{} that
	$$\P_2D_i = D_i, \ i=1,2, \qquad \text{where} \qquad \P_2(x_1,x_2,x_3) = (-x_1,-x_2,x_3).$$
	We then have the following result on the capacitance coefficients. 
	\begin{lemma} \label{lem:P2sym}
		Assume that $\P_2D_i=D_i, \ i=1,2$.
		\begin{itemize}
			\item[(i)] $\psi_j^0$, and consequently $C_{ij}^0$ and $\mathbf{c}_j$, are independent of $\alpha_0$.
			\item[(ii)] For some $c\in \R$ we have
			$$ \mathbf{c}_1 = \left(\begin{smallmatrix}0 \\ 0 \\ c\end{smallmatrix}\right), \quad \mathbf{c}_2 = \left(\begin{smallmatrix}0 \\ 0 \\ -c\end{smallmatrix}\right),$$
		and
		$$C^{1,\alpha_0} = -\frac{\iu w_3 L^2}{2} \begin{pmatrix} 1 & 1 \\ 1 & 1\end{pmatrix} - \frac{\iu w_3 c^2}{2L^2}\begin{pmatrix} 1 & -1 \\ -1 & 1\end{pmatrix} + O(\omega).$$
		\end{itemize}
	\end{lemma}
	\begin{proof}[Proof of (i)]
		Using the symmetries described by $\P$ and by $\P_2$, and using the fact that $\psi_1^0 = -\psi_2^0$ we have
		\begin{equation}\label{eq:psisym}
		\psi_j^{0}(y) = -\psi_j^{0}(\P\P_2y), \qquad \hat{\psi}_1^{1,\alpha_0}(y) = \hat{\psi}_2^{1,\alpha_0}(\P\P_2y),		
		\end{equation}
		for $j=1,2$. Using the first identity, we have for $i=1,2$,
		\begin{equation} \label{eq:c1,2}
		\mathbf{e}_i \cdot \mathbf{c}_j =\int_{\D}\P\P_2(y_i) \psi_j^{0}(\P\P_2 y)\de\sigma(y) = -\int_{\D}y_i \psi_j^{0}\de\sigma = -\mathbf{e}_i \cdot \mathbf{c}_j,
		\end{equation}
		where $\mathbf{e}_i$ is the $i$\textsuperscript{th} standard vector.
		Using the second identity of \eqref{eq:psisym}, we have
		\begin{equation}\label{eq:intpsihat}
		\int_{\D} \hat{\psi}_1^{1,\alpha_0} \dx \sigma = \int_{\D} \hat{\psi}_2^{1,\alpha_0} \dx \sigma.		
		\end{equation}
		Using \eqref{eq:c1,2} and \eqref{eq:intpsihat}, we find from \eqref{eq:leadingorder} that
		$$\S_D^{0,0}[\psi_1^0] = \frac{1}{2}\Chi_{\D_1} - \frac{1}{2}\Chi_{\D_2}, \qquad \S_D^{0,0}[\psi_2^0] = -\frac{1}{2}\Chi_{\D_1} + \frac{1}{2}\Chi_{\D_2}.$$
		Since $\S_D^{0,0}$ is injective on $L^2_0(\D)$, and $\S_D^{0,0}$ does not depend on $\alpha_0$, we conclude that $\psi_j^0$ does not depend on $\alpha_0$.
		\end{proof}
		
		\begin{proof}[Proof of (ii)]
		In the proof of \Cref{lem:cap0} it was proved that $\psi_j^{0}$, and hence $\mathbf{c}_j$, is real-valued. The first statement of \textit{(ii)} now follows from \eqref{eq:c1,2} and the fact that $\mathbf{c}_1 = -\mathbf{c}_2$.	
		
		Next, we prove the second statement of \textit{(ii)}. Since $\hat \S_D^{\omega \alpha_0,\omega}[\psi_i^{\omega\alpha_0,\omega}] = \Chi_{\p D_i}$ we use the asymptotic expansions in \Cref{sec:C0} to conclude that 
		\begin{align*}
		\S_D^{0,0}[\hat{\psi}_1^{1,\alpha_0}- \hat{\psi}_2^{1,\alpha_0}]	&= K \Chi_ {\D} - \frac{\alpha\cdot(x_1,x_2)}{2k_3L^2} \int_{\D} \left(\hat{\psi}_1^{1,\alpha_0}- \hat{\psi}_2^{1,\alpha_0}\right) \dx \sigma \\
		&= K \Chi_ {\D},
		\end{align*}
		for some constant $K$, where we have used \eqref{eq:intpsihat} in the last step. From \Cref{lem:phi=0} it follows that 	
		$$\hat{\psi}_1^{1,\alpha_0} = \hat{\psi}_2^{1,\alpha_0}.$$
		Then, using \Cref{lem:C1}, we can write the matrix $C^{1,\alpha_0}$ as
		$$C^{1,\alpha_0} = -\frac{\iu w_3 L^2}{2} \begin{pmatrix} 1 & 1 \\ 1 & 1\end{pmatrix} + h\begin{pmatrix} 1 & -1 \\ -1 & 1\end{pmatrix} + O(\omega),$$
		where
		$$h = \int_{\D}\S_1^{\alpha_0}[\psi_1^{0}] \psi_1^{0} \dx \sigma.$$
		The only remaining task is to explicitly compute $h$. To this end, we write the kernel function $G_1^{\alpha_0}$ of $\S_1^{\alpha_0}$ as
		$$G_1^{\alpha_0}(x) = K_1(x) + K_2^{\alpha_0}(x) + K_3^{\alpha_0}(x),$$
		and hence
		$$h = \int_{\D}\int_{\D} \big(K_1(x-y) + K_2^{\alpha_0}(x-y) +  K_3^{\alpha_0}(x-y)  \big)\psi_1^0(x)\psi_1^0(y) \dx\sigma(x)\dx\sigma(y),$$
		where
		$$K_1(x) = \frac{\iu w_3x_3^2}{4L^2}, \quad K_2^{\alpha_0}(x) = \alpha_0\cdot\left( \frac{\iu|x_3|(x_1,x_2)}{2L^2} + g_1(x)\right), \quad K_3^{\alpha_0}(x) = \frac{\iu \left(\alpha_0\cdot (x_1,x_2)\right)^2}{4w_3L^2}.$$
		Next, we will show that only $K_1$ gives a non-zero contribution to $h$. Firstly, we observe that $K_2^{\alpha_0}(\P_2x) = - K_2^{\alpha_0}(x)$ while $\psi_1^0(\P_2 x) = \psi_1^0(x)$. Therefore 
		$$\int_{\D}\int_{\D} K_2^{\alpha_0}(x-y)\psi_1^0(x)\psi_1^0(y) \dx\sigma(x)\dx\sigma(y) = 0.$$
		Secondly, we study the contribution of $K_3$. We have
		$$K_3^{\alpha_0}(x-y) = \frac{\iu}{4w_3L^2}\left( \left(\alpha_0\cdot (x_1,x_2)\right)^2 - 2\left(\alpha_0\cdot (x_1,x_2)\right)\left(\alpha_0\cdot (y_1,y_2)\right) + \left(\alpha_0\cdot (y_1,y_2)\right)^2 \right),$$
		and hence
		\begin{multline*}
		\int_{\D}\int_{\D} K_3^{\alpha_0}(x-y)\psi_1^0(x)\psi_1^0(y) \dx\sigma(x)\dx\sigma(y) = \\ \frac{\iu}{4w_3L^2}\left(\int_{\D}\left(\alpha_0\cdot (x_1,x_2)\right)^2 \psi_1^0(x)\dx \sigma(x) \int_{\D}\psi_1^0\dx\sigma\right. 
		+\int_{\D}\left(\alpha_0\cdot (y_1,y_2)\right)^2 \psi_1^0(y)\dx \sigma(y) \int_{\D}\psi_1^0\dx\sigma\\ -2\left.\int_{\D}\alpha_0\cdot(x_1,x_2) \psi_1^0(x)\dx\sigma(x)\int_{\D}\alpha_0\cdot(y_1,y_2) \psi_1^0(y)\dx\sigma(y)\right).
		\end{multline*}
		The first two terms in the right-hand side vanish since $\int_{\D}\psi_1^0\dx\sigma = 0$, while the last term vanishes since $\mathbf{e}_i \cdot \mathbf{c}_1 = 0$ for $i=1,2$. We conclude that only $K_1$ has a non-zero contribution to $h$. We have 
		$$K_1(x-y) = \frac{\iu w_3}{4L^2}\left(x_3^2 - 2x_3y_3 + y_3^2\right),$$
		so analogously to $K_3$, we can use the fact that $\int_{\D}\psi_1^0\dx\sigma = 0$ to conclude that 
		\begin{align*}
		h = \int_{\D}\int_{\D} K_1(x-y)\psi_1^0(x)\psi_1^0(y) \dx\sigma(x)\dx\sigma(y) &= -\frac{\iu w_3}{2L^2}\int_{\D}x_3 \psi_1^0(x)\dx\sigma(x)\int_{\D}y_3 \psi_1^0(y)\dx\sigma(y) \\
		&=-\frac{\iu w_3 c^2}{2L^2}.
		\end{align*}
		This proves the claim.
	\end{proof}
	\begin{thm} \label{thm:unidir}
		Assume that $\P_2D_i=D_i$, for $i=1,2$, and that $|bL^2| \neq |a c| $. Let $\lambda = \omega^2|D_1|$, and assume that $\omega$ is in the subwavelength regime such that $\lambda = \lambda_2^{0} + \lambda^*$ for $\lambda^* = O(\omega^3)$.	We then have the following asymptotic expansion of the scattering matrix:
		$$S = \frac{1}{\iu k_3\left(\frac{b^2L ^2}{a} - \frac{a c^2}{L^2}\right) - \lambda^*}\begin{pmatrix} \lambda^* - 2 k_3 bc & \iu k_3\left(\frac{b^2L ^2}{a} + \frac{a c^2}{L^2} \right) \\ \iu k_3\left(\frac{b^2L ^2}{a} + \frac{a c^2}{L^2} \right) & \lambda^* + 2 k_3 bc \end{pmatrix} + O(\omega),$$	
		where the error term is uniform with respect to $\lambda^*$ in a neighbourhood of $0$. In particular, we have $r_+ \neq r_-$, and to leading order $r_+$ and $r_-$ vanish, respectively, at $\lambda^* = \lambda_+$ and $\lambda^* = \lambda_-$ given by
		$$\lambda_+ = 2 k_3 bc, \quad \lambda_- = -2 k_3 bc.$$
	\end{thm}
	\begin{proof} 
	We begin by computing $p$. We have
	\begin{align*} \label{eq:q}
	\int_{\D_i}\left({\S}_D^{\alpha,k}\right)^{-1}[u^{\mathrm{in}}]\de\sigma &= \int_{\D}u^{\mathrm{in}}\left({\S}_D^{-\alpha,k}\right)^{-1}[\Chi_{\p D_i}]\de\sigma = \int_{\D}\iu \k\cdot x \psi_i^{0} \de\sigma + \omega\int_{\D} \psi_i^{1,-\alpha_0}\de\sigma + O(\omega^2) \nonumber\\ &= \iu k_3 \left(L^2 - (-1)^{i}sc\right)   + O(\omega^2), \end{align*}
	where, as before, $s$ denotes the sign of the third component of $\k$. Then we find that 
	$$d^\mathrm{T} p = -2\iu k_3\left( sac + \iu bL^2\right) + O(\omega^4).$$
	Moreover, writing $f(\lambda^*) = d^\mathrm{T}\left(\omega C^{1,v,\alpha_0} - \lambda^*I\right)\v_2^{0}$ we have
	$$f(\lambda^*) =  - 2\iu k_3 \left( b^2 L^2  - \frac{a^2c^2}{L^2}\right) +2a\lambda^* + O(\omega^6).$$	
	We can then compute $G$ as
	\begin{align*}
	G_{s\sigma}(\lambda^*) = -\frac{2k_3}{ f(\lambda^*)L^2}\left( s a c  +\iu  bL ^2\right)\left( \sigma a c -\iu bL ^2\right) - 1 + O(\omega).
	\end{align*}
	Then, to leading order we have 
	$$G_{++}(\lambda^*) = G_{--}(\lambda^*), \qquad G_{+-}(\lambda^*) - G_{-+}(\lambda^*) = -\frac{8k_3 abc }{ f(\lambda^*)}.$$	
	If $b\neq 0$, it is clear that $r_+ \neq r_-$. Simplifying these expressions, we have 
	$$G_{\pm\mp}(\lambda^*) = \frac{-2a}{ f(\lambda^*)}\left(\lambda^* \pm 2k_3 bc \right)+ O(\omega).$$	
	To leading order, we then have the following expressions for $t_\pm$ and $r_\pm$:
	\begin{align*}
	r_+ &= \frac{\lambda^* - 2 k_3 bc}{\iu k_3\left(\frac{b^2L ^2}{a} - \frac{a c^2}{L^2}\right) - \lambda^*}, & t_- &= \frac{\iu k_3\left( \frac{b^2L ^2}{a} + \frac{a c^2}{L^2}\right)}{\iu k_3\left(\frac{b^2L ^2}{a} - \frac{a c^2}{L^2}\right)  - \lambda^*}, \\[0.3em]
	t_+ &= \frac{\iu k_3\left(\frac{b^2L ^2}{a} + \frac{a c^2}{L^2} \right)}{\iu k_3\left(\frac{b^2L ^2}{a} - \frac{a c^2}{L^2}\right) - \lambda^*}, & r_- &= \frac{ \lambda^* + 2 k_3 bc }{\iu k_3\left(\frac{b^2L ^2}{a} - \frac{a c^2}{L^2}\right) - \lambda^*}.
	\end{align*}
	The expression for $S$ and the zeros of $r_\pm$ follow directly from this.
\end{proof}

\begin{remark}
	There are two subwavelength frequency regimes not covered in \Cref{thm:unidir}: when $\omega$ is close to the first band at $\omega_1^0 = 0$ or when $\omega$ is well-separated from the two bands. When $\omega$ is close to $\omega_1^0$, \Cref{rmk:first} tells us that $t_+ = t_- = 1$ and $r_+=r_- = 0$. When $\omega$ is well-separated from $\omega_1^0$ and $\omega_2^\alpha$, the solution $q$ to \eqref{eq:eval_C_quasi_short} will be small. Consequently, it is easy to show that $t_+$  and $t_-$ will be small, while $r_+$ and $r_-$ have magnitude close to $1$. These regimes are demonstrated in \Cref{fig:scattering}.
\end{remark}
\begin{remark}
	The assumption $|bL^2| \neq |a c|$ comes from the condition $\Im(d^\mathrm{T} C^{v,1,\alpha_0} \v_2^{0}) \neq 0$ in \Cref{lem:modal_quasi}. At the critical point $b= \pm\frac{ac}{L^2}$, the denominator of $S$ will vanish at $\lambda^* = 0$. Around this point, we therefore expect the transmittance and reflectance to be very large, corresponding to extraordinary transmission. This is demonstrated numerically in \Cref{fig:extraordinary}.
\end{remark}

\begin{remark}
	Throughout this section, we use the classical convention for the scattering matrix $S$, defined in \eqref{eq:S_def}. If we instead define $S = S(\lambda^*)$ by 
	$$
	S = \begin{pmatrix}  t_+ & r_- \\  r_+ & t_- \end{pmatrix}.
	$$
	we see that the points $\lambda^* = \lambda_+$ and $\lambda^* = \lambda_-$ represent exceptional points of $S$ (see, for example, \cite{huang2017unidirectional} for further elaborations on the connection between unidirectional reflection and exceptional points).
\end{remark}

\begin{remark}
	In the case that ${b} = 0$, \ie{} without gain and loss, it is well known that $r_+ = r_-$ and $t_+ = t_-$, which is consistent with the fact that $G_{+-} = G_{-+}$ and $G_{++} = G_{--}$ in this case.
\end{remark}

\subsubsection{Numerical illustration}
	\begin{figure}
	\begin{center}
		\includegraphics[width=0.7\linewidth]{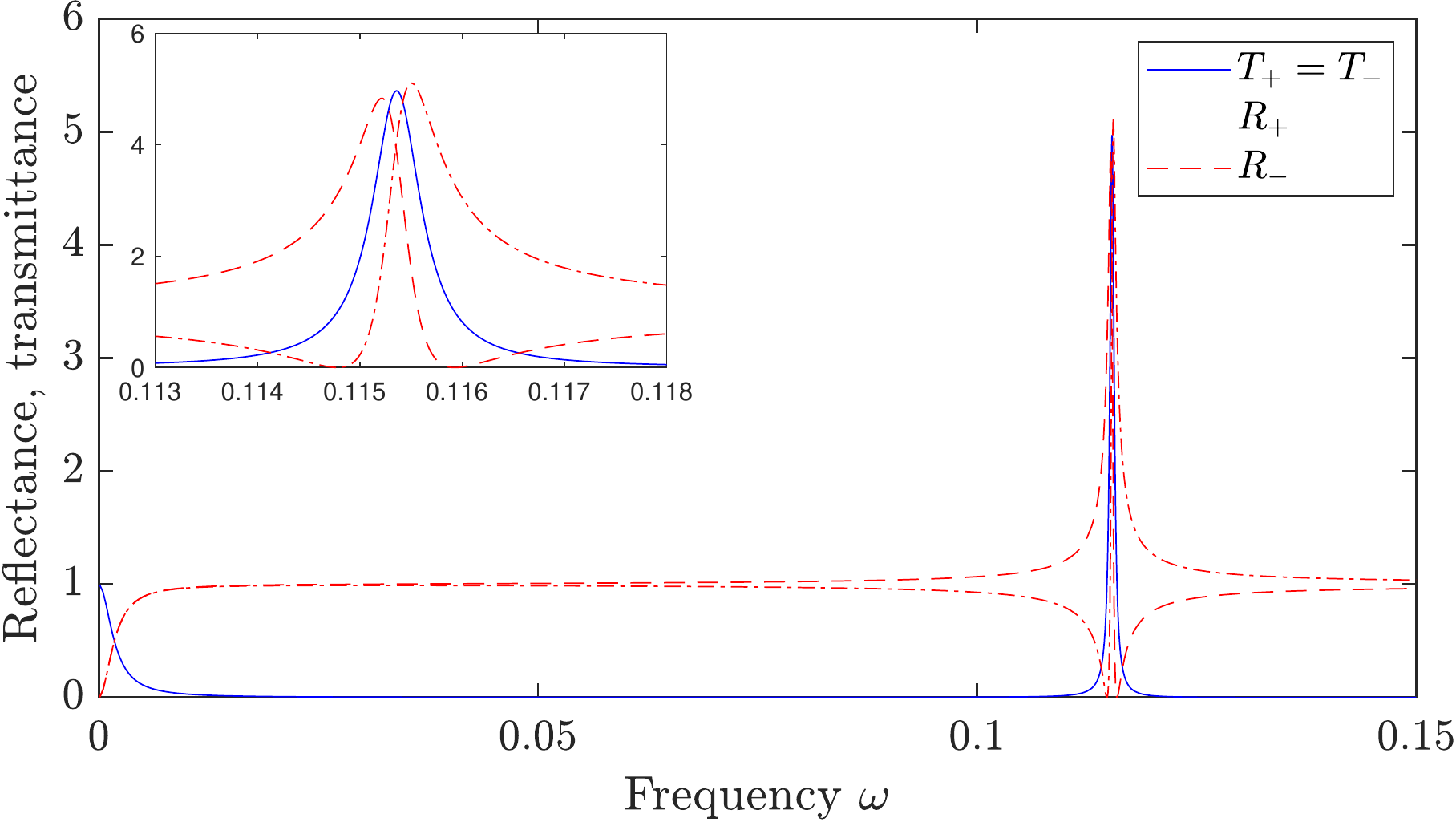}
	\end{center}
	\caption{Plot of the transmittance $T_\pm = |t_\pm|^2$ (blue) and reflectance $R_\pm = |r_\pm|^2$ (red) as functions of the frequency. The inlay shows the behaviour around the critical frequency range and demonstrates both unidirectional reflection and extraordinary transmission. Here, we simulate a two-dimensional problem with the same parameters and the same frequency range as \Cref{fig:bandexceptional}, with incident direction $\w = \frac{1}{2}\left(-\sqrt{3}, \ \pm 1\right)^\mathrm{T}$.} \label{fig:scattering}
\end{figure}
\Cref{fig:scattering} shows the transmittance $T_\pm = |t_\pm|^2$ and reflectance $R_\pm = |r_\pm|^2$ as functions of the frequency. The computations were performed using the multipole discretization (see, for example, \cite{ammari2018mathematical}), independently of the asymptotic analysis in the previous subsections. As is well known for $\mathcal{PT}$-symmetric structures (see \emph{e.g.} \cite{yuan2019unidirectional}), the two transmission coefficients $t_+$ and $t_-$ coincide. The figure clearly shows the shifted zeros of the reflectances close to the second resonant frequency. For a frequency at one of these zeros, the system will exhibit unidirectional reflectionless transmission.

Due to the gain and loss, the reflectance and transmittance satisfy the ``generalized'' energy conservation relation \cite{li2012conservation}
$$R_+R_- + 2\sqrt{T_+T_-} - T_+T_- = 1,$$
which, in particular, allows the scattering matrix to be non-unitary and allows the reflectance or transmittance to exceed 1. In \Cref{fig:extraordinary}, the peak transmittance is plotted as a function of the gain/loss parameter $b$, which clearly demonstrates the extraordinary transmission.
	\begin{figure}[h]
	\begin{center}
		\includegraphics[width=0.7\linewidth]{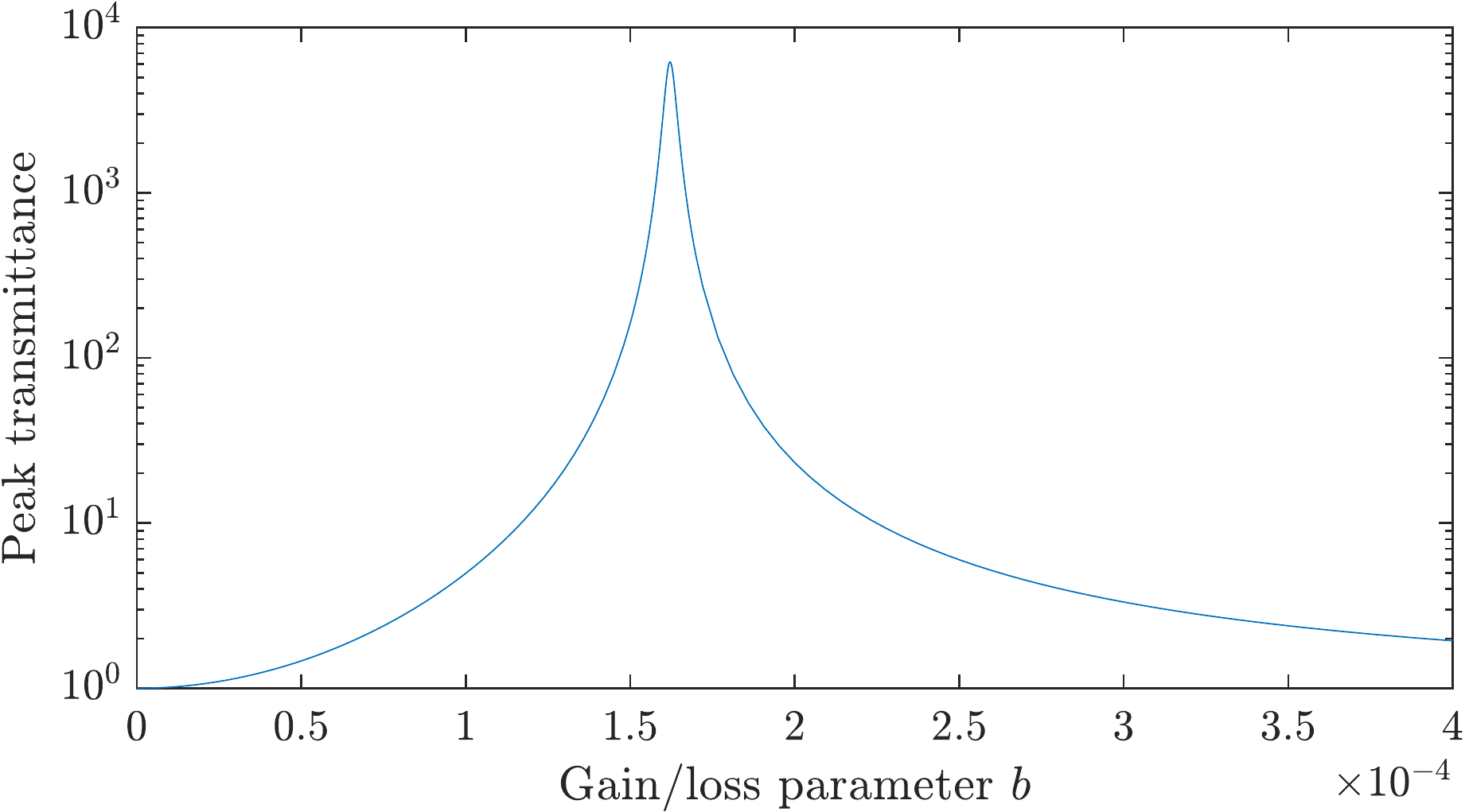}
	\end{center}
	\caption{Plot of the peak transmittance as a function of the gain/loss parameter. The extraordinarily high transmittance at $b = \left| \frac{ac}{L^2}\right|$ is clearly demonstrated. Here, we simulate a two-dimensional problem with the same parameters as \Cref{fig:bandexceptional} and \Cref{fig:scattering}.} \label{fig:extraordinary}
\end{figure}

	\section{Resonator cavities} \label{sec:cavities}
	In this section, we examine the properties of finite metamaterials taking the form of cavities filled with a large number of small subwavelength resonators with non-real material parameters. While the pair of high-contrast resonators in \Cref{sec:exceptional} interacts with wavelengths much larger than their size, we would like to design these cavities so that they might exhibit similar exceptional behaviour in response to wavelengths of the same order as their dimensions. We study this system using a homogenization approach, deriving the effective equations as the size of the resonators becomes small and the number of resonators becomes large.

	\subsection{Homogenization of non-Hermitian cavities} \label{sec:symm_cavities}
	
	\begin{figure}
		\centering
		\begin{tikzpicture}
		%%% omega 2
		\draw (-0.1,0) ellipse (1.6 and 1.5);
		\node at (-1.7,-1.2){\large $\Omega_+$};
		\draw circle (0.09) node{\tiny $+$};
		\begin{scope}[xshift=0.6cm,yshift=0.6cm]
		\draw circle (0.09) node{\tiny $+$};
		\end{scope}
		\begin{scope}[xshift=-0.3cm,yshift=0.5cm]
		\draw circle (0.09) node{\tiny $+$};
		\end{scope}
		\begin{scope}[xshift=-0.5cm,yshift=-0.5cm]
		\draw circle (0.09) node{\tiny $+$};
		\end{scope}
		\begin{scope}[xshift=0.4cm,yshift=-0.6cm]
		\draw circle (0.09) node{\tiny $+$};
		\end{scope}
		\begin{scope}[xshift=-0.8cm,yshift=0.1cm]
		\draw circle (0.09) node{\tiny $+$};
		\end{scope}
		%%% omega 2
		\begin{scope}[xshift=5cm]
		\draw (-0.1,0) ellipse (1.6 and 1.5);
		\node at (1.7,-1.2){\large $\Omega_-$};
		\begin{scope}[xshift=0.4cm]
		\draw circle (0.09) node{\tiny $-$};
		\end{scope}
		\begin{scope}[xshift=0.1cm,yshift=-0.5cm]
		\draw circle (0.09) node{\tiny $-$};
		\end{scope}
		\begin{scope}[xshift=-0.3cm,yshift=0.6cm]
		\draw circle (0.09) node{\tiny $-$};
		\end{scope}
		\begin{scope}[xshift=0.5cm,yshift=0.5cm]
		\draw circle (0.09) node{\tiny $-$};
		\end{scope}
		\begin{scope}[xshift=-0.4cm,yshift=-0.6cm]
		\draw circle (0.09) node{\tiny $-$};
		\end{scope}
		\begin{scope}[xshift=-0.8cm,yshift=0.1cm]
		\draw circle (0.09) node{\tiny $-$};
		\end{scope}
		\end{scope}
		\end{tikzpicture}
		\caption{A pair of $\mathcal{PT}$-symmetric cavities of many small resonators. Here, $+$ and $-$ denote opposite signs of the imaginary part of the material coefficients.} \label{fig:PT_cavities}
	\end{figure}
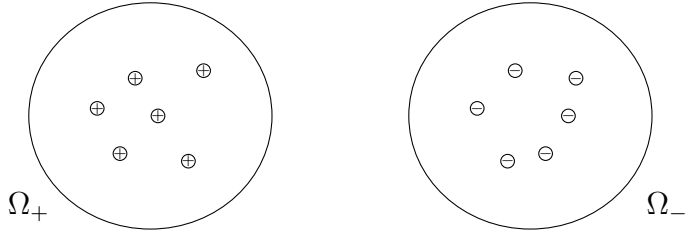
	We first derive a version of \Cref{lem:modal} which describes how an asymptotically small resonator $D_0^r=r D_0+z$ (where $D_0$ is some fixed, connected domain) scatters an incoming field. So that the resonant frequencies are of order 1, we will assume that if the size of the resonator $r\to0$ then the material parameters of its interior are given by
	\begin{equation} \label{ass:parameter_scaling}
	v_0^2\delta_0:=r^2a+\iu r^{2+\epsilon_1}b,
	\end{equation}
	for some fixed $0<\epsilon_1<1$ and real-valued constants $a, b=O(1)$. We will fix $a>0$ and consider the cases $b>0$ (gain on each small resonator) and $b<0$ (loss on each small resonator) separately. We study the scattering problem
	\begin{equation} \label{eq:scattering_singleres}
	\left\{
	\begin{array} {ll}
	\ds \Delta {u}+ k^2 {u}  = 0 & \text{in } \R^3 \setminus \overline{D_0^r}, \\[0.3em]
	\ds \Delta {u}+ k_0^2 {u}  = 0 & \text{in } D_0^r, \\
	\nm
	\ds  {u}|_{+} -{u}|_{-}  = 0  & \text{on } \partial D_0^r, \\
	\nm
	\ds  \delta_0 \frac{\partial {u}}{\partial \nu} \bigg|_{+} - \frac{\partial {u}}{\partial \nu} \bigg|_{-} = 0 & \text{on } \partial D_0^r, \\
	\nm
	\ds u(x) - u^{\mathrm{in}}(x) & \text{satisfies the Sommerfeld radiation} \\ & \text{condition as }  |x| \rightarrow \infty, 
	\end{array}
	\right.
	\end{equation}
	where $\delta_0\ll1$ and $k_0=\omega/v_0$. 
	
	\begin{lemma} \label{lem:point_scatter}
	Let $D_0\subset\mathbb{R}^3$ be some fixed resonator (whose boundary satisfies $\p D_0 \in C^{1,s}$ for some $0 < s < 1$) and define the small resonator $D_0^r$, for some small $r>0$, as
	\begin{equation*}
	D_0^r=r D_0+z,
	\end{equation*}
	where $z\in\mathbb{R}^3$ is the new centre of $D_0^r$. Assume that the material parameters within $D_0^r$ satisfy \eqref{ass:parameter_scaling} and that $\omega^2-(\omega^*)^2=Cr^{\epsilon_1}$ for some fixed $C\in\mathbb{C}$, where
	\begin{equation*}
	(\omega^*)^2=\frac{(a+\iu r^{\epsilon_1}b)\mathrm{Cap}_{D_0}}{|D_0|}.
	\end{equation*}
	As $r\rightarrow 0$, the solution to the Helmholtz problem \eqref{eq:scattering_singleres} for scattering by $D_0^r$ can be written as
	\begin{equation*} \label{eq:point_scatter_one}
	u(x)-u^{\mathrm{in}}(x)=r\mathrm{Cap}_{D_0}\, \frac{\omega^2}{\omega^2-(\omega^*)^2} G^k(x-z) u^{\mathrm{in}}(z)+O(r^{2-\epsilon_1}).
	\end{equation*}
	\end{lemma}
	
	\begin{proof}
	The solutions to the scattering problem can be represented as 
	\begin{equation*} %\label{eq:layer_potential_representation}
	u = \begin{cases}
	u^{\mathrm{in}}(x)+\S_{D_0^r}^{k}[\psi](x), & x\in\R^3 \setminus \overline{D_0^r},\\
	\S_{D_0^r}^{k_1}[\phi](x), & x\in D_0^r,
	\end{cases}
	\end{equation*} 
	where $k_1 = \omega/v_1$, for some surface potentials $(\phi,\psi)\in L^2(\partial D_0^r)\times L^2(\partial D_0^r)$, which must be chosen so that $u$ satisfies the transmission conditions across $\partial D_0^r$. 
	
	We wish to replicate \Cref{lem:modal} in the present setting, using asymptotic expansions in terms of $r\ll1$ (and $\delta=O(r^2)$), while $\omega=O(1)$. We have, as $r\to0$, that 
	\begin{gather*} 
	\S_{D_0^r}[\phi-\psi]=u^{\mathrm{in}}+O(r) \quad \text{on} \ \p D_0^r, %\label{eq:first}
	\\
	\left(-\frac{1}{2}I+\K_{D_0^r}^*+\frac{\omega^2}{v_1^2}\K_{{D_0^r},2}\right)[\phi]-\delta_0\left(\frac{1}{2}I+\K_{D_0^r}^*\right)[\psi]=O(r^2) \quad \text{on} \ \p D_0^r.
	\end{gather*}
	Repeating the arguments of \Cref{lem:modal}, we find that the solution to the scattering problem can be written as 
	\begin{equation*}
	u-u^{\mathrm{in}} = q S_{D_0^r}^\omega - \S_{D_0^r}^{k}\left[\S_{D_0^r}^{-1}[u^{\mathrm{in}}]\right] + O(r),
	\end{equation*}
	where
	\begin{equation*}
	S_{D_0^r}^\omega(x) = \begin{cases}
	\S_{D_0^r}^{k}\left[\S_{D_0^r}^{-1}[\Chi_{\D_0^r}]\right](x), & x\in \mathbb{R}^3\setminus \overline{D_0^r},\\
	\S_{D_0^r}^{k_1}\left[\S_{D_0^r}^{-1}[\Chi_{\D_0^r}]\right](x), & x\in D_0^r,
	\end{cases} 
	\end{equation*}
	and $q=q(\omega)$ satisfies
	\begin{equation*}
	\left(-\omega^2|D_0^r|-v_0^2\delta_0 \,\int_{\D_0^r}\S_{D_0^r}^{-1}[\Chi_{\D_0^r}]\de\sigma\right)q=-v_0^2\delta_0\int_{\D_0^r} \S_{D_0^r}^{-1}[u^{\mathrm{in}}] \de\sigma +O(r^4).
	\end{equation*}
	Let
	\begin{equation*}
	\mathrm{Cap}_{D_0}:=-\int_{\D_0}\S_{D_0}^{-1}[\Chi_{\D_0}]\de\sigma,
	\end{equation*}
	then we have that
	\begin{gather*}
	\int_{\D_0^r}\S_{D_0^r}^{-1}[\Chi_{\D_0^r}]\de\sigma=-r\mathrm{Cap}_{D_0}, \qquad
	\int_{\D_0^r}\S_{D_0^r}^{-1}[u^{\mathrm{in}}] \de\sigma = -r\mathrm{Cap}_{D_0} u^{\mathrm{in}}(z)+O(r^2), \nonumber \\
	S_{D_0^r}^\omega= -r\mathrm{Cap}_{D_0}G^k(x-z)+O(r^2), \label{eq:scaling_properties}\\
	\S_{D_0^r}^k\left[\S_{D_0^r}^{-1}[u^{\mathrm{in}}]\right]=-r\mathrm{Cap}_{D_0}u^{\mathrm{in}}(z)G^k(x-z)+O(r^2), \nonumber
	\end{gather*}
	from which the result follows.
	\end{proof}
	
	We now wish to consider a spherical domain $\Omega$ which contains a (large) number of small, identical resonators (\emph{e.g.} $\Omega_+$ or $\Omega_-$ in \Cref{fig:PT_cavities}). If $D_0$ is a fixed domain, then for some $r>0$ the $N$ resonators are given, for $1\leq j\leq N$, by
	\begin{equation*}
	D_{0,j}^{r,N}=r D_0+z_j^N,
	\end{equation*}
	for positions $z_j^N$. We will always assume that $r$ is sufficiently small such that the resonators are not overlapping and that $D_{0}^{r,N}=\bigcup_{j=1}^N D_{0,j}^{r,N} \Subset \Omega$. We choose the number of resonators $N$ so that there exists some positive number $\Lambda$ such that
	\begin{equation} \label{ass:number}
	r^{1-\epsilon_1}N=\Lambda.
	\end{equation}
	The choice of $\Lambda$ will be an important subtlety in the major theorem of this section.

	We will find the effective equation in the specific case that the frequency $\omega=O(1)$ satisfies
	\begin{equation} \label{ass:frequency}
	\omega^2=\frac{a\mathrm{Cap}_{D_0}}{|D_0|}.
	\end{equation} 
	In this case, we are able to use a result from \cite{ammari2017effective} which says that, since the resonators are small, we can use the point-scatter approximation from \Cref{lem:point_scatter} to describe how they interact with incoming waves. To do so, we must make some extra assumptions on the regularity of the distribution $\{z_j^N:1\leq j\leq N\}$ so that the system is well behaved as $N\to\infty$ (under the assumption \eqref{ass:number}). In particular, we assume that there exists some constant $\eta$ such that for any $N$ it holds that
	\begin{equation} \label{ass:dist}
	\min_{i\neq j} |z_i^N-z_j^N| \geq \frac{\eta}{N^{1/3}},
	\end{equation}
	and, further, there exists some $0<\epsilon_0<1$ and constants $C_1,C_2>0$ such that for all $h\geq 2\eta N^{-1/3}$,
	\begin{align}
	\sum_{|x-z_j^N|\geq h} \frac{1}{|x-y_j^N|^2}\leq C_1 N |h|^{-\epsilon_0}, \qquad&\text{uniformly for all } x\in\Omega,\label{ass:reg1}\\
	\sum_{2\eta N^{-1/3}\leq|x-z_j^N|\leq 3h} \frac{1}{|x-y_j^N|}\leq C_2 N |h|, \qquad&\text{uniformly for all } x\in\Omega. \label{ass:reg2}
	\end{align}
	Finally, we will also need that 
	\begin{equation} \label{ass:epsilon}
	\epsilon_2:=\frac{\epsilon_1}{1-\epsilon_1}-\frac{\epsilon_0}{3}>0.
	\end{equation}
	If we represent the field that is scattered by the collection of resonators $D_0^{r,N}=\bigcup_{j=1}^N D_{0,j}^{r,N}$ as
	\begin{equation*} %\label{eq:layer_potential_representation}
	u^N(x)= \begin{cases}
	u^{\mathrm{in}}(x)+\S_{D_0^{r,N}}^{k}[\psi^N](x), & x\in\R^3 \setminus \overline{D_0^{r,N}},\\
	\S_{D_0^{r,N}}^{k_0}[\phi^N](x), & x\in D_0^{r,N},
	\end{cases}
	\end{equation*} 
	for some $\psi^N, \phi^N\in L^2(\p D_0^{r,N})$,	then we have the following lemma, which follows from \cite[Proposition~3.1]{ammari2017effective}. This justifies using a point-scatter approximation to describe the total incident field acting on the resonator $D_{0,j}^{r,N}$ and the scattered field due to $D_{0,j}^{r,N}$, defined respectively as
	\begin{equation*}
	u_j^{in,N}=u^{\mathrm{in}}+\sum_{i\neq j} \S_{D_{0,i}^{r,N}}^k[\psi^N] \qquad \text{and} \qquad 	u_j^{s,N}=\S_{D_{0,j}^{r,N}}^k[\psi^N].
	\end{equation*}

	\begin{lemma} \label{lem:points}
		Under the assumptions \eqref{ass:frequency}--\eqref{ass:epsilon}, it holds that the total incident field acting on the resonator $D_{0,j}^{r,N}$ is given, at $z_j^N$, by
		\begin{equation*}
		u_j^{in,N}(z_j^N)=u^{\mathrm{in}}(z_j^N)+\sum_{i\neq j} r\mathrm{Cap}_{D_0}\, \frac{\omega^2}{\omega^2-(\omega^*)^2} G^k(z_j^N-z_i^N)u^{\mathrm{in}}(z_j^N),
		\end{equation*}
		up to an error of order $O(N^{-\epsilon_2})$. Similarly, it holds that the scattered field due to the resonator $D_{0,j}^{r,N}$ is given, at $x$ such that $|x-z_j^N|\gg r$, by
		\begin{equation*}
		u_j^{s,N}(x)= r\mathrm{Cap}_D\, \frac{\omega^2}{\omega^2-(\omega^*)^2} G^k(x-z_j^N) u_j^{in,N}(z_j^N) ,
		\end{equation*}
		up to an error of order $O(N^{-\epsilon_2}+r|x-z_j^N|^{-1})$.
	\end{lemma}
	
	In order for the sums in \Cref{lem:points} to be well behaved as $N\to\infty$, we make one additional assumption on the regularity of the distribution: that there exists a real-valued function $\widetilde{V}\in C^1(\overline{\Omega})$ such that for any $f\in C^{0,\alpha}(\Omega)$, with $0<\alpha\leq1$, there is a constant $C_3$ such that %for any $N$ it holds that
	\begin{equation} \label{ass:integral}
	\max_{1\leq j\leq N} \left| \frac{1}{N}\sum_{i\neq j} G^k(z_j^N-z_i^N)f(z_i^N)-\int_\Omega G^k(z_j^N-y)\widetilde{V}(y)f(y) \de y \right| \leq C_3\frac{1}{N^{\alpha/3}}\|f\|_{C^{0,\alpha}(\Omega)}.
	\end{equation}
	\begin{remark}
		It will hold that $\widetilde{V}\geq0$. If the resonators' centres $\{z_j^N:j=1,\dots,N\}$ are uniformly distributed, then $\widetilde V$ will be a positive constant, $\widetilde V = \frac{1}{|\Omega|}.$
	\end{remark}

	Under all these assumptions, we are able to derive effective equations for the system with an arbitrarily large number of small resonators. If we let $\epsilon_3\in(0,\tfrac{1}{3})$, then we will seek effective equations on the set given by
	\begin{equation*}
	Y_{\epsilon_3}^N:=\left\{x\in\mathbb{R}^3:|x-z_j^N|\geq N^{\epsilon_3-1} \text{ for all }1\leq j\leq N\right\},
	\end{equation*}
	which is the set of points that are sufficiently far from the resonators, so avoid the singularities of the Green's function.
	
	\begin{thm} \label{thm:homogenized}
		Under the assumptions \eqref{ass:number}--\eqref{ass:integral}, the solution $u^N$ to the scattering problem \eqref{eq:scattering_singleres} with the system of resonators $D_0^{r,N}=\bigcup_{j=1}^N D_{0,j}^{r,N}$ converges to the solution of
		\begin{equation*}
		\begin{cases}
		\left(\Delta+k^2-\frac{\iu\Lambda a\mathrm{Cap}_D}{b}\widetilde{V}(x)\right)u(x)=0, &  x\in\Omega, \\
		\left(\Delta+k^2\right)u(x)=0, & x\in\mathbb{R}^3\setminus\Omega,\\
		u\big|_- = u \big|_+ & \mbox{ on } \partial \Omega,
		\end{cases}
		\end{equation*}
		as $N\to\infty$, together with a radiation condition governing the behaviour in the far field, which says that uniformly for all $x\in Y_{\epsilon_3}^N$ it holds that
		\begin{equation*}
		|u^N(x)-u(x)|\leq C N^{-\min\left\{\frac{1-\epsilon_0}{6},\epsilon_2,\epsilon_3,\frac{1-\epsilon_3}{3}\right\}}.
		\end{equation*}
		If $a>0$ and $b<0$, then this convergence holds regardless of the choice of $\Lambda$. If $b>0$, then there exists at least one $\Lambda\in\mathbb{R}$ for which the solution converges.
	\end{thm}
	\begin{proof}
		This follows by modifying the results of \cite{ammari2017effective}. Much of this is straightforward, the important subtlety being to show that the operator
		$$\mathcal{T}[f](x):=\frac{\iu\Lambda a\mathrm{Cap}_{D_0}}{b}\int_\Omega G^k(x-y) \widetilde{V}(y) f(y) \de y,$$
		is such that $I-\mathcal{T}$ is invertible. Since $\mathcal{T}$ is compact, $I-\mathcal{T}$ is of Fredholm type so is invertible if and only if it is injective. Consider, first, the case that $a>0$ and $b<0$ and suppose $f\in C^{0,\alpha}(\Omega)$ is such that $(I-\mathcal{T})[f]=0$. Applying $\Delta+k^2$, we see that
		\begin{equation*}
		\Delta f = -k^2 f + \frac{\iu \Lambda a \mathrm{Cap}_{D_0}}{b} \widetilde{V} f \quad \text{in }\Omega,
		\end{equation*}
		from which we see that
		\begin{equation} \label{eq:f_neg}
			-\int_{\Omega} |\nabla f|^2 \de x + \int_{\partial\Omega} \ddp{f}{\nu} \overline{f} \de\sigma = -k^2\int_{\Omega} |f|^2 \de x + \frac{\iu \Lambda a \mathrm{Cap}_{D_0}}{b} \int_{\Omega} \widetilde{V} |f|^2 \de x.
		\end{equation}
		From \cite{nedelec2001acoustic} we know that
		\begin{equation} \label{eq:im_sign}
		\Im\int_{\partial\Omega} \ddp{f}{\nu}\overline{f} \de\sigma\geq0,
		\end{equation}
		with equality only if $f=0$. Since $\widetilde{V}\geq0$ we have also that $\int_{\Omega} \widetilde{V} |f|^2 \de x\geq0$ so taking the imaginary part of \eqref{eq:f_neg} gives us that
		\begin{equation*}
		0\leq\Im\int_{\partial\Omega} \ddp{f}{\nu}\overline{f} \de\sigma =  \frac{ \Lambda a \mathrm{Cap}_{D_0}}{b} \int_{\Omega} \widetilde{V} |f|^2 \de x \leq 0,
		\end{equation*}
		hence $f=0$. 
		 
		 Conversely, if $b>0$ then we must take more care to choose the constant $\Lambda$ to guarantee invertibility. Assume, for contradiction, that we cannot choose $\Lambda$ such that $I-\mathcal{T}$ is invertible. Then, we can choose a sequence of real numbers $\{\Lambda_n:n\in\mathbb{N}\}$ such that $\Lambda_n\to0$ and for each $n$ there exists $0\neq f_n\in H^1(\Omega)$ such that $(I-\mathcal{T})f_n=0$. Hence, it holds that $g_n:={f_n}/{\|f_n\|_{H^1(\Omega)}}$ satisfies
		\begin{equation} \label{eq:gn}
		\begin{cases}
		\Delta g_n + k^2 g_n - \frac{\iu\Lambda_n a\mathrm{Cap}_{D_0}}{b} \widetilde{V} g_n=0 & \text{in } \Omega, \\
		\ddp{g_n}{\nu}=\mathcal{N}_k(g_n) & \text{on } \partial\Omega,
		\end{cases}
		\end{equation}
		where $\mathcal{N}_k$ is the Dirichlet-to-Neumann map on the exterior of $\Omega$, defined as $\mathcal{N}_k[\varphi]:=\ddp{v}{\nu}\big|_{\partial\Omega}$ where $v$ solves $(\Delta +k^2)v=0$ on $\mathbb{R}^3\setminus\overline{\Omega}$ with $v=\varphi$ on $\partial\Omega$ and the Sommerfeld radiation condition at infinity. Since $\{g_n:n\in\mathbb{N}\}$ is bounded in $H^1(\Omega)$, which is compactly embedded into $L^2(\Omega)$, there exists some $g\in H^1(\Omega)$ such that (passing to a subsequence) %$g_n\rightharpoonup g$ in $H^1(\Omega)$ and
		$g_n\to g$ in $L^2(\Omega)$. 
		
		We want to show that, in fact, $g_n$ converges strongly to $g$ in $H^1(\Omega)$ and that $g=0$, which will contradict the fact that $\|g_n\|_{H^1{\Omega}}=1$ for all $n$. Studying the limiting form of \eqref{eq:gn}, we see that the limit $g$ is the restriction of $w$ to $\Omega$, where $w$ is the solution to $(\Delta+k^2)w=0$ on $\mathbb{R}^3$ with the Sommerfeld radiation condition at infinity. This is well known to have a unique solution given by $w=0$, hence $g=w|_\Omega=0$. Analogous to \eqref{eq:f_neg}, it holds for each $n$ that 
		\begin{equation*} \label{eq:f_pos}
		-\int_{\Omega} |\nabla g_n|^2 \de x + \int_{\partial\Omega} \ddp{g_n}{\nu}\overline{g_n} \de\sigma = -k^2\int_{\Omega} |g_n|^2 \de x + \frac{\iu \Lambda_n a \mathrm{Cap}_{D_0}}{b} \int_{\Omega} \widetilde{V} |g_n|^2 \de x,
		\end{equation*}
		and we know from \cite{nedelec2001acoustic} that
		\begin{equation} \label{eq:re_sign}
		\Re\int_{\partial\Omega} \ddp{g_n}{\nu}\overline{g_n} \de\sigma\leq0,
		\end{equation}
		so we see that
		\begin{equation*}
		\int_{\Omega} |\nabla g_n|^2 \de x \leq k^2\int_{\Omega} |g_n|^2 \de x.
		\end{equation*}
		Therefore, $\nabla g_n\to0$ in $L^2(\Omega)$ so we have that $g_n\to 0$ in $H^1(\Omega)$, which gives the desired contradiction.

		Once we know that $I-\mathcal{T}$ is invertible, we can see that the limiting system is well posed and the rest of the argument (in particular, proving that the field given by \Cref{lem:points} converges to the limiting system) follows from \cite{ammari2017effective}.
	\end{proof}

	\begin{remark}
		The assumption \eqref{ass:frequency} is important so that the frequency $\omega$ is close to the resonant frequency $\omega^*$. In particular, the difference is such that $\omega^2-(\omega^*)^2=O(r^{\epsilon_1})$. This means that the behaviour will be dominated by the monopole resonant modes of each small resonator. If we relaxed this assumption, then other coupled modes might be excited, invalidating the use of the point-interaction approximation from \Cref{lem:points}.
	\end{remark}
	\begin{remark}
		The assumption that $\Omega$ is spherical is needed so that we are able to infer \eqref{eq:im_sign} and \eqref{eq:re_sign} from the results of \cite{nedelec2001acoustic}. 
	\end{remark}

	\subsection[Effective medium for PT-symmetric pairs]{Effective medium for $\mathcal{PT}$-symmetric pairs} \label{sec:pt_cavities}
	\begin{figure}
		\centering
		\begin{tikzpicture}
		\draw (-0.1,0) ellipse (2.2 and 1.5);
		\node at (1.75,-1.2){\large $\Omega$};
		\begin{scope}[xshift=1cm,rotate=10]
		\draw (-0.12,0) circle (0.09) node[rotate=10]{\tiny $+$};
		\draw (0.12,0) circle (0.09) node[rotate=10]{\tiny $-$};
		\end{scope}
		\begin{scope}[xshift=-1.1cm,yshift=0.5cm,rotate=50]
		\draw (-0.12,0) circle (0.09) node[rotate=50]{\tiny $+$};
		\draw (0.12,0) circle (0.09) node[rotate=50]{\tiny $-$};
		\end{scope}
		\begin{scope}[xshift=0.4cm,yshift=-0.6cm,rotate=-70]
		\draw (-0.12,0) circle (0.09) node[rotate=-70]{\tiny $-$};
		\draw (0.12,0) circle (0.09) node[rotate=-70]{\tiny $+$};
		\end{scope}
		\begin{scope}[xshift=0.3cm,yshift=0.7cm,rotate=-110]
		\draw (-0.12,0) circle (0.09) node[rotate=-110]{\tiny $+$};
		\draw (0.12,0) circle (0.09) node[rotate=-110]{\tiny $-$};
		\end{scope}
		\begin{scope}[xshift=-0.8cm,yshift=-0.5cm,rotate=-180]
		\draw (-0.12,0) circle (0.09) node[rotate=-180]{\tiny $+$};
		\draw (0.12,0) circle (0.09) node[rotate=-180]{\tiny $-$};
		\end{scope}
		\begin{scope}[xshift=-0.3cm,yshift=0.12cm,rotate=160]
		\draw (-0.12,0) circle (0.09) node[rotate=160]{\tiny $+$};
		\draw (0.12,0) circle (0.09) node[rotate=160]{\tiny $-$};
		\end{scope}
		\end{tikzpicture}
		\caption{A cavity containing many small $\mathcal{PT}$-symmetric pairs of resonators. Here, $+$ and $-$ denote opposite signs of the imaginary part of the material coefficients. Microscopic $\mathcal{PT}$-symmetry does not lead to useful symmetry at the macroscale, since this property is lost under homogenization.} \label{fig:PT_cavity}
	\end{figure}
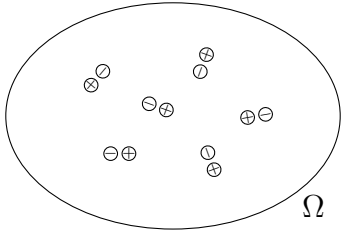
	
	It is interesting to compare the results of \Cref{sec:symm_cavities} to the corresponding behaviour of a cavity filled with a large collection of small $\mathcal{PT}$-symmetric pairs of resonators, as depicted in \Cref{fig:PT_cavity}. {This setting does not provide an approach to designing non-Hermitian cavities, so we avoid doing any rigorous homogenization in this case.} However, interesting behaviour is seen when each pair is poised at an asymptotic exceptional point (\emph{cf.} similar analysis of the real-valued case in \cite{ammari2017double}).
	
	Recall the $\mathcal{PT}$-symmetric resonator pair $D=D_1\cup D_2$ from \Cref{sec:exceptional}. We will define the small dimer $D^r=D_1^r\cup D_2^r$, for some small $r>0$, as
	\begin{equation*}
	D^r=r D+z,
	\end{equation*}
	where $z\in\mathbb{R}^3$ is the new centre of $D^r$. We re-use the notation for the material parameters from \Cref{sec:exceptional} but, in order for resonance to occur at $O(1)$ frequencies, scale the material parameters so that
	\begin{equation} \label{eq:parameter_choice}
	v_1^2\delta_1:=r^2a+\iu r^{2}b, \qquad v_2^2\delta_2:=r^2a-\iu r^{2}b,
	\end{equation}
	for real-valued constants $a,b=O(1)$. In this case, we have chosen both the real and imaginary parts of $v_i^2\delta_i$ to be $O(r^2)$ since they need to have the same asymptotic behaviour in order for the resonator pair to support an asymptotic exceptional point, as predicted by \Cref{thm:exceptional}.
	
	We must first replicate \Cref{lem:modal} in the present setting, using asymptotic expansions in terms of $r\ll1$ (and $\delta=O(r^2)$), while $\omega=O(1)$. We have, as $r\to0$, that the solution to the problem \eqref{eq:scattering} for scattering by $D^r$ can be represented in the form \eqref{eq:layer_potential_representation} with densities $\phi,\psi\in L^2(\D^r)\times L^2(\D^r)$ which satisfy
	\begin{gather*} 
	\S_{D^r}[\phi-\psi]=u^{\mathrm{in}}+O(r), \quad \text{on} \ \p D_1^r\cup\p D_2^r, %\label{eq:first}
	\\
	\left(-\frac{1}{2}I+\K_{D^r}^*+\frac{\omega^2}{v_j^2}\K_{{D^r},2}\right)[\phi]-\delta_j\left(\frac{1}{2}I+\K_{D^r}^*\right)[\psi]=O(r^2), \quad \text{on} \ \p D_j^r,\, j=1,2.
	\end{gather*}
	Repeating the arguments of \Cref{lem:modal}, we find that the solution to the scattering problem can be written as
	\begin{equation} \label{eq:excep_decomp} 
	u-u^{\mathrm{in}} = q_1S_{D^r,1}^\omega + q_2S_{D^r,2}^\omega - \S_{D^r}^\omega\left[\S_{D^r}^{-1}[u^{\mathrm{in}}]\right] + O(r),
	\end{equation}
	where 
	$$S_{D^r,j}^\omega(x) = \begin{cases}
	\S_{D^r}^{k}\left[\S_{D^r}^{-1}[\Chi_{\D_j^r}]\right](x), & x\in\mathbb{R}^3\setminus\overline{D^r},\\[0.3em]
	\S_{D^r}^{k_i}\left[\S_{D^r}^{-1}[\Chi_{\D_j^r}]\right](x), & x\in D_i^r, \ i=1,2,
	\end{cases} 
	$$
	and the constants $q_1$ and $q_2$ satisfy
	\begin{equation} \label{eq:q_excep}
	\left(C_{D^r}^v-\omega^2|D_1^r|I\right)\begin{pmatrix}q_1\\q_2\end{pmatrix}
	=-
	\begin{pmatrix} r^2(a+\iu b)\int_{\D_1^r}\S_{D^r}^{-1}[u^{\mathrm{in}}]\de\sigma \\[0.3em]
	r^2(a-\iu b)\int_{\D_2^r}\S_{D^r}^{-1}[u^{\mathrm{in}}]\de\sigma \end{pmatrix}
	+O(r^4).
	\end{equation}
	
	We now wish to compute expressions for $q_1$ and $q_2$ in the case that we are at the asymptotic exceptional point, meaning that $b=b_0$ as specified by \Cref{thm:exceptional}. In this case, $C_{D^r}^v$ is non-Hermitian and has one eigenvalue with a one-dimensional eigenspace. We will use the Jordan decomposition for $C_{D^r}^v$. Using the notation $C_{ij}$ to denote the capacitance coefficients of the original fixed dimer $D$, as defined in \Cref{sec:exceptional}, the eigenvalue of $C_{D^r}^v$ is given by $r^3\lambda_1$ where $\lambda_1=aC_{11}$. We have that
	\begin{equation} \label{eq:JNF}
	C_{D^r}^v=SJS^{-1},
	\end{equation}
	where
	\begin{equation*}
	J = \begin{pmatrix}
	r^3\lambda_1 & 1 \\ 0 & r^3\lambda_1
	\end{pmatrix}, \quad
	S=\begin{pmatrix}
	-rC_{12} & \frac{\iu C_{12}}{r^2b_0 C_{11}} \\
	\frac{\iu rb_0 C_{11}}{a+ib_0} & 0
	\end{pmatrix}, \quad 
	S^{-1}=-\begin{pmatrix}
	0 & \frac{\iu(a+\iu b_0)}{rb_0 C_{11}} \\
	\frac{\iu r^2b_0 C_{11}}{C_{12}} & r^2(a+ib_0)
	\end{pmatrix}.
	\end{equation*}
	Using \eqref{eq:JNF} and writing $\lambda=r^{-3}\omega^2|D_1^r|=\omega^2|D_1|$, the formula \eqref{eq:q_excep} gives us that
	\begin{equation*}
	\begin{pmatrix}q_1\\q_2\end{pmatrix}
	=- S\begin{pmatrix} r^{-3}(\lambda_1 - \lambda)^{-1} & -r^{-6}(\lambda_1 - \lambda)^{-2} \\ 0 & r^{-3}(\lambda_1 - \lambda)^{-1} \end{pmatrix}S^{-1}\begin{pmatrix} 
	r^2(a+\iu b)\int_{\D_1^r}\S_{D^r}^{-1}[u^{\mathrm{in}}]\de\sigma \\[0.3em] 
	r^2(a-\iu b)\int_{\D_2^r}\S_{D^r}^{-1}[u^{\mathrm{in}}]\de\sigma \end{pmatrix}+O(r),
	\end{equation*}
	\emph{i.e.}
	\begin{equation} \label{eq:excep_coeffs}
	\begin{pmatrix}q_1\\q_2\end{pmatrix} = r^{-1}
	\begin{pmatrix} Q_{11} & Q_{12} \\ Q_{21} & Q_{22} \end{pmatrix}
	\begin{pmatrix} (a+\iu b)\int_{\D_1^r}\S_{D^r}^{-1}[u^{\mathrm{in}}]\de\sigma \\[0.3em] (a-\iu b)\int_{\D_2^r}\S_{D^r}^{-1}[u^{\mathrm{in}}]\de\sigma \end{pmatrix}+O(r),
	\end{equation}
	where
	\begin{align*}
	Q_{11}&={\iu b_0 C_{11}}\frac{1}{(\lambda - \lambda_1)^2}+\frac{1}{\lambda - \lambda_1}, \qquad & 
	Q_{12}&={C_{12}(a+\iu b_0)}\frac{1}{(\lambda - \lambda_1)^2}, \\
	Q_{21}&= \frac{b_0^2 C_{11}^2}{(a+\iu b_0)C_{12}}\frac{1}{(\lambda - \lambda_1)^2}, &
	Q_{22}&=-{\iu b_0 C_{11}}\frac{1}{(\lambda - \lambda_1)^2}+\frac{1}{\lambda - \lambda_1}.
	\end{align*}

	\begin{lemma} \label{lem:excep_dimers}
	As $r\rightarrow 0$, the solution to the Helmholtz problem \eqref{eq:scattering} for scattering by the small $\mathcal{PT}$-symmetric dimer $D^r=rD+z$ with fixed frequency $\omega=O(1)$ can be written as
	\begin{equation*} 
	u(x)-u^{\mathrm{in}}(x)=r\, {m}(\omega)G^k(x)u^{\mathrm{in}}(0)+O(r^2),
	\end{equation*} 
	where, if $\omega_1=\sqrt{aC_{11}|D_1|^{-1}}$, 
	\begin{equation*}
	{m}(\omega)=\mathrm{Cap}_D\left(\frac{a^2C_{11}C_{12}}{|D_1|^2}\frac{1}{(\omega^2-\omega_1^2)^2} +\frac{a\mathrm{Cap}_D}{2|D_1|}\frac{1}{\omega^2-\omega_1^2}+1\right).
	\end{equation*}
	\end{lemma}
	\begin{proof}
	The terms in \eqref{eq:excep_decomp} and \eqref{eq:excep_coeffs} can be further simplified using scaling properties analogous to \eqref{eq:scaling_properties}. Note that, thanks to the assumed symmetry $\P D = D$, it holds that $\mathrm{Cap}_D=2(C_{11}+C_{12})$. Then, we have that 
	\begin{gather*}
	\int_{\D_j^r}\S_{D^r}^{-1}[u^{\mathrm{in}}] \de\sigma = -r\tfrac{1}{2}\mathrm{Cap}_D u^{\mathrm{in}}(0)+O(r^2), \nonumber \\
	S_{D^r,j}^\omega(x)= -r\tfrac{1}{2}\mathrm{Cap}_DG^k(x)+O(r^2), \label{eq:scaling_properties_dimer}\\
	\S_{D^r}^k\left[\S_{D^r}^{-1}[u^{\mathrm{in}}]\right](x)=-r\mathrm{Cap}_Du^{\mathrm{in}}(0)G^k(x)+O(r^2). \nonumber \qedhere
	\end{gather*} 
	\end{proof}

	\begin{remark}
	It is interesting to consider using \Cref{lem:excep_dimers} as the starting point for a similar homogenization argument to the one we applied to cavities of single resonators in Section \ref{sec:exceptional}. Define $N$ small resonator pairs as $D_j^N=r D+z_j^N$. Allowing a formal argument (and assuming all the required conditions to guarantee \emph{e.g.} the validity of the point-scatter approximation and the convergence of the microfield to the effective one), we observe that as $N\to\infty$ we should obtain the homogenized equation
	\begin{equation*}
	\begin{cases}
	\left(\Delta+k^2-\Lambda {m}(k/v) V(x)\right)u(x)=0, &  x\in\Omega, \\
	\left(\Delta+k^2\right)u(x)=0, & x\in\mathbb{R}^3\setminus\Omega,\\
	u\big|_- = u \big|_+ & \mbox{ on } \partial \Omega,
	\end{cases}
	\end{equation*}
	where $m$ is specified in \Cref{lem:excep_dimers} and $V$ is a function that depends on the resonators' positions. Both $m$ and $V$ are real valued, meaning this effective equation has purely real parameters.
	\end{remark}

\section{Concluding remarks}
In this work, we have studied non-Hermitian systems of high-contrast subwavelength resonators with parity--time symmetry. We have proved the existence of asymptotic exceptional points in a system of two resonators. More precisely, we have proved that there is a value of the gain/loss parameter such that the resonant frequencies and eigenmodes coincide at leading order (in terms of the material contrast). Moreover,  we have studied a metascreen of $\mathcal{PT}$-symmetric resonators. We proved that the two reflection coefficients asymptotically vanish at distinct frequencies, which allows for unidirectional reflectionless transmission. Additionally, we have showed that extraordinarily high transmittance can occur for a specific magnitude of the gain/loss.
Finally, we have proved that large ensembles of non-Hermitian resonators collectively behave as non-Hermitian systems, meaning they might, for example, support exceptional points on a macroscopic scale. Our results in this work are fundamental to understanding not only wave scattering by systems of non-Hermitian subwavelength resonators but also Fano-type resonances in periodic systems.

\bibliographystyle{abbrv}
\bibliography{exceptional}{}

\end{document}